\documentclass{article}
\usepackage{euscript,amsmath,amssymb,amsfonts,graphicx,bm,amsthm,mathtools}
\usepackage{cite}

  \usepackage{paralist}
  \usepackage{graphics} 
\usepackage[caption=false]{subfig}
\usepackage{soul}

\usepackage{latexsym,amssymb,enumerate,amsmath,amsthm} 
  \usepackage{graphicx} 
  \usepackage{tikz}
     \usepackage[colorlinks=false]{hyperref}
 \hypersetup{urlcolor=blue, citecolor=red}
\usepackage{latexsym,amssymb,enumerate,amsmath} 
\usepackage{pgfplots}
\usepackage{soul,xcolor}

\newcommand{\dd}{\textup{d}}
\def\eps{\varepsilon}
\def\E{\mathbb{E}}
\def\P{\mathbb{P}}
\def\R{\mathbb{R}}

\def\os{\tau}
\def\t{\widetilde{\tau}}
\def\wt{\widetilde{\tau}}
\def\ot{\tau}

\newtheorem{theorem}{Theorem}
\newtheorem{lemma}[theorem]{Lemma}
\newtheorem{proposition}[theorem]{Proposition}

\theoremstyle{plain}
\theoremstyle{remark}

\theoremstyle{definition}

\begin{document}


\title{Extreme statistics of superdiffusive L{\'e}vy flights and every other L{\'e}vy subordinate Brownian motion}


\author{Sean D. Lawley\thanks{Department of Mathematics, University of Utah, Salt Lake City, UT 84112 USA (\texttt{lawley@math.utah.edu}).}
}
\date{\today}
\maketitle

\begin{abstract}
The search for hidden targets is a fundamental problem in many areas of science, engineering, and other fields. Studies of search processes often adopt a probabilistic framework, in which a searcher randomly explores a spatial domain for a randomly located target. There has been significant interest and controversy regarding optimal search strategies, especially for superdiffusive processes. The optimal search strategy is typically defined as the strategy that minimizes the time it takes a given single searcher to find a target, which is called a first hitting time (FHT). However, many systems involve multiple searchers and the important timescale is the time it takes the fastest searcher to find a target, which is called an extreme FHT. In this paper, we study extreme FHTs for any stochastic process that is a random time change of Brownian motion by a L{\'e}vy subordinator. This class of stochastic processes includes superdiffusive L{\'e}vy flights in any space dimension, which are processes described by a Fokker-Planck equation with a fractional Laplacian. We find the short-time distribution of a single FHT for any L{\'e}vy subordinate Brownian motion and use this to find the full distribution and moments of extreme FHTs as the number of searchers grows. We illustrate these rigorous results in several examples and numerical simulations.
\end{abstract}

\section{Introduction}

What is the best way to search for a target whose location is \emph{a priori} unknown? This basic search problem arises at various spatial and temporal scales in many areas of science, engineering, and other fields \cite{benichou2011rev}. Examples include rescuers searching for castaways \cite{frost2001}, military forces searching for enemy targets \cite{morse1956},  animals searching for food, shelter, or a mate \cite{reynolds2018, benichou2011rev, viswanathan2008}, proteins searching for DNA binding sites \cite{lomholt2005}, and computers searching a database \cite{kao1996}. 

Empirical and theoretical studies of search processes often adopt a probabilistic framework, in which the searcher randomly explores a spatial domain for a randomly located target \cite{shlesinger2006, benichou2011rev, viswanathan2008}. The random movement of the searcher is often classified as diffusive, subdiffusive, or superdiffusive, depending respectively on whether the square of its displacement scales linearly, sublinearly, or superlinearly in time. While subdiffusive and superdiffusive motion are termed ``anomalous'' diffusion, they have been observed in many physical and biological systems \cite{metzler2004}.

Mathematical models of random search often assume that searchers explore space via a continuous-time random walk. In this framework, a searcher waits at its current location for a random time chosen from some waiting time probability density $w(t)$, and then moves to a new location by jumping a random distance chosen from some jump length probability density $l(y)$. The searcher repeats these two steps indefinitely or until it reaches the target. This process can be diffusive, subdiffusive, or superdiffusive, depending on the tails of the waiting time density $w(t)$ and the jump length density $l(y)$. In particular, if the mean waiting time is finite and the jump length density has the following slow power law decay,
\begin{align}\label{pl}
l(y)
\propto y^{-1-\alpha}\quad\text{as $y\to\infty$ for some $\alpha\in(0,2)$},
\end{align}
then the process is superdiffusive and is often called a L{\'e}vy flight \cite{metzler2004, dubkov2008}. 
In a certain scaling limit, the probability density ${{p}}(x,t)$ for the position of a L{\'e}vy flight in $\R^{d}$ satisfies the space fractional Fokker-Planck equation \cite{meerschaert2019}.
\begin{align}\label{ffpe0}
\partial_{t}{{p}}=-{{K}}(-\Delta)^{\alpha/2}{{p}},\quad x\in\R^{d},\,t>0,
\end{align}
where $(-\Delta)^{\alpha/2}$ denotes the fractional Laplacian \cite{lischke2020} and ${{K}}>0$ is the generalized diffusion coefficient. 
Similar models of superdiffusive search involving long relocation events with distances chosen from a power law density akin to \eqref{pl} are L{\'e}vy walks \cite{zaburdaev2015} and truncated L{\'e}vy flights \cite{viswanathan2008}.

Many have argued that superdiffusion is a more efficient search method compared to normal diffusion, since superdiffusive processes spend less time in previously explored regions of space \cite{shlesinger1986, viswanathan1996, viswanathan1999}. Signatures of superdiffusion have been found in movement data for many different animal species \cite{reynolds2018}, including albatrosses \cite{viswanathan1996}, spider monkeys \cite{ramos2004, boyer2006}, jackals \cite{atkinson2002}, sharks \cite{sims2006}, microorganisms \cite{bartumeus2003}, and also within biological cells \cite{reverey2015}. In addition, superdiffusive search methods involving L{\'e}vy flights are employed in computational algorithms such as simulated annealing \cite{pavlyukevich2007}. However, there has been controversy regarding the empirical evidence of superdiffusion in animal foraging, as some have questioned the accuracy of the statistical methods used in some of these studies \cite{edwards2007}.

Furthermore, the theoretical optimality of superdiffusive search models has also been controversial. Indeed, the seminal work of Viswanathan et al.\ \cite{viswanathan1999} in 1999 claimed that the search time of a single searcher is minimized if the searcher employs an inverse square L{\'e}vy walk (corresponding to $\alpha=1$ in \eqref{pl}). This result forms the core of the very influential L{\'e}vy flight foraging hypothesis, which states that biological organisms must have evolved to perform such L{\'e}vy walks because of their optimality \cite{viswanathan2008}. However, a recent analysis proved that this founding result of the L{\'e}vy flight foraging hypothesis is incorrect \cite{levernier2020, buldyrev2021, levernier2021reply}. 

Search time is often quantified in terms of a so-called first hitting time (FHT), which is the first time the searcher reaches the target. If $X=\{X(t)\}_{t\ge0}$ denotes the position of the searcher as a function of time $t\ge0$, then the FHT is
\begin{align}\label{fht}
\tau
:=\inf\{t>0:X(t)\in U\},
\end{align}
where $U$ denotes the position of the target(s) (or the region of space in which the searcher can detect the target). There have been many studies of FHTs of L{\'e}vy flights, using both computational and analytical approaches \cite{eliazar2004, koren2007, koren2007first, gao2014, palyulin2019, wardak2020}. An interesting aspect of these studies is the discrepancy between first hitting events (the searcher reaches or hits the target) and first passage events (the searcher moves beyond the target). While these two notions are equivalent for standard diffusion processes, paths of L{\'e}vy flights are discontinuous and thus may jump across a target without actually hitting it, which is called a ``leapover'' \cite{koren2007, koren2007first, palyulin2019, wardak2020}.

Studies of optimal search strategies generally ask what search method minimizes the FHT of a single searcher \cite{shlesinger2006, benichou2011rev, viswanathan2008}. However, many systems involve multiple searchers and the relevant timescale is the time it takes the fastest searcher to find the target. Indeed, this can be the case for many of the traditional search scenarios referenced above, such as the search for missing persons and castaways, military searches for enemy targets, and computer search processes. In the context of biology, cellular events are often triggered when the first of many searchers finds a target \cite{schuss2019}, and cooperative foraging involves multiple animals working together to find a target \cite{schoener1971, traniello1977, holldobler1990, wenzel1991, jarvis1998, torney2009, torney2011, feinerman2012}. If $\tau_{1},\dots,\tau_{N}$ are the respective FHTs of $N$ parallel searchers, then the fastest searcher finds the target at time
\begin{align}\label{TN}
T_{N}
:=\min\{\tau_{1},\dots,\tau_{N}\},
\end{align}
which is often called an extreme FHT, fastest FHT \cite{lawley2020esp1}, or parallel FHT \cite{ro2017, clementi2020}. Hence, in these scenarios the relevant question is not what search strategy minimizes the single searcher FHT $\tau$ in \eqref{fht}, but rather what search strategy minimizes the extreme FHT $T_{N}$ in \eqref{TN}.

In this paper, we investigate the extreme FHTs of a general class of stochastic processes which includes superdiffusive L{\'e}vy flights in $\R^{d}$ for any $d\ge1$. The class of stochastic processes is called L{\'e}vy subordinate Brownian motions, as the processes are obtained by random time changes of Brownian motion. We find the short-time distribution of a single FHT $\tau$ and then use this to find the full distribution and moments of the extreme FHT $T_{N}$ as the number of searchers $N$ grows.

\begin{figure}
  \centering
             \includegraphics[width=0.465\textwidth]{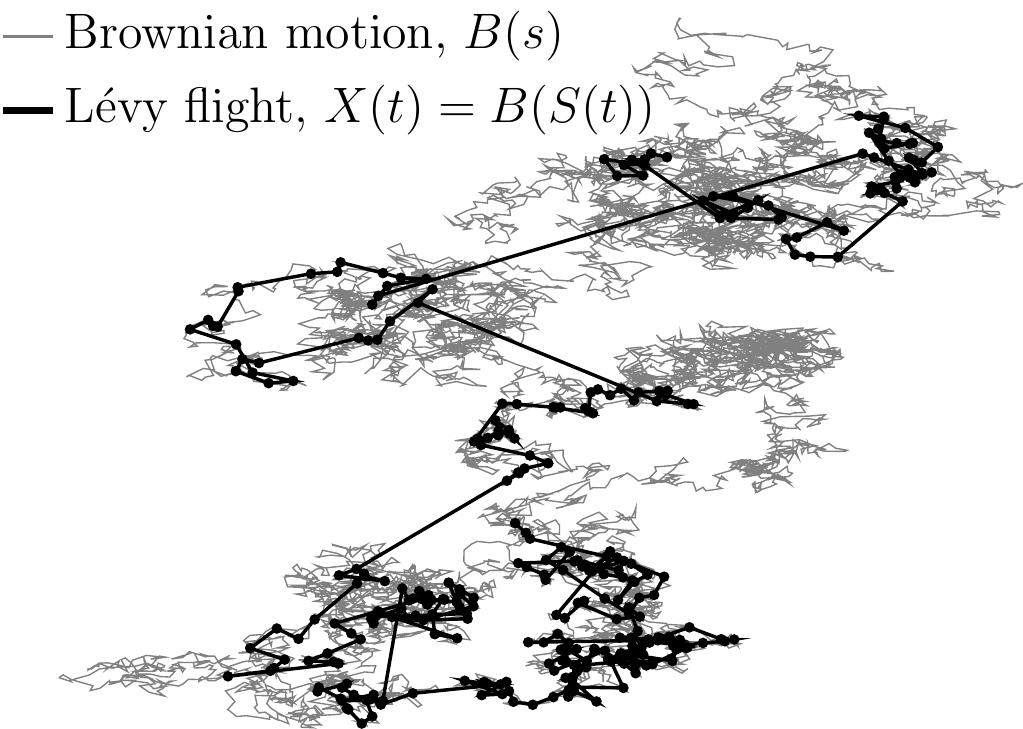}
                              \qquad
               \includegraphics[width=0.465\textwidth]{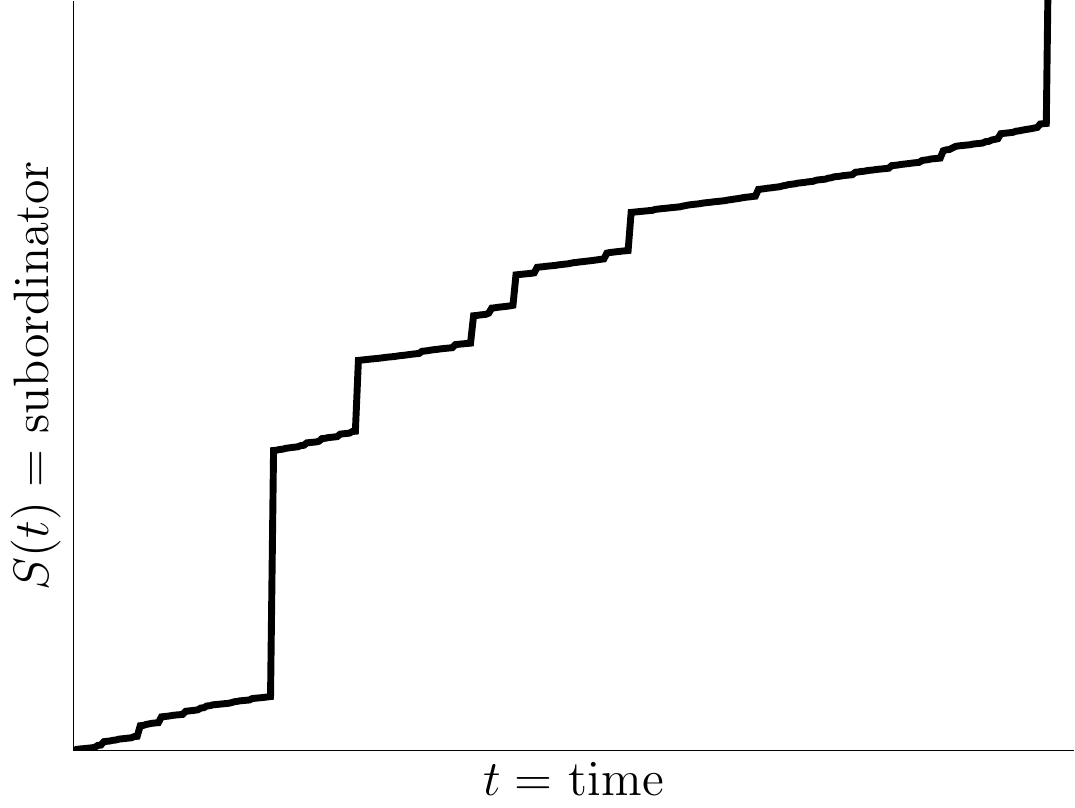}
 \caption{L{\'e}vy flights are subordinate Brownian motions. Left: The thin gray trajectory is the path of a Brownian motion $B(s)$, and the thick black trajectory is the path of a L{\'e}vy flight $X(t)=B(S(t))$ obtained as a random time change of $B(s)$ according to the $(\alpha/2)$-stable subordinator $S(t)$. Right: The path of the $(\alpha/2)$-stable subordinator $S(t)$. We take $\alpha=1.5$ in this plot.}
 \label{figschem}
\end{figure}

To summarize our results, let ${{B}}=\{{{B}}(s)\}_{s\ge0}$ be a $d$-dimensional Brownian motion with unit diffusivity, and let $S=\{S(t)\}_{t\ge0}$ be an independent subordinator, which means that $S$ is a one-dimensional, nondecreasing L{\'e}vy process with $S(0)=0$. Define the path of a single searcher $X=\{X(t)\}_{t\ge0}$ by
\begin{align}\label{tc0}
X(t)
:={{B}}(S(t))+X(0)\in\R^{d},\quad t\ge0,
\end{align}
where $X(0)$ is some initial position independent of ${{B}}$ and $S$. That is, $X$ is a random time change of Brownian motion (see Figure~\ref{figschem} for the special case that $X$ is a L{\'e}vy flight). Assuming that $X(0)$ cannot lie in the target $U$, we prove that the FHT in \eqref{fht} has the universal short-time distribution,
\begin{align}\label{st0}
\P(\tau\le t)
\sim\P(X(t)\in U)
\sim \rho t\quad\text{as }t\to0+,
\end{align}
where $\rho\in(0,\infty)$ is the rate,
\begin{align}\label{rho0}
\rho
:=\int_{0}^{\infty}\P({{B}}(s)+X(0)\in U)\,\nu(\dd s),
\end{align}
and $\nu(\dd s)$ is the L{\'e}vy measure of $S$. Throughout this paper, ``$f\sim g$'' means $f/g\to1$. If we set $X(0)=0$, then the Gaussianity of ${{B}}(s)$ means the integrand in \eqref{rho0} is
\begin{align*}
\P({{B}}(s)+X(0)\in U)
=\frac{1}{(4\pi s)^{d/2}}\int_{U}\exp\Big(\frac{-\|x\|^{2}}{4s}\Big)\,\dd x.
\end{align*}
We prove \eqref{st0} for any nondeterministic subordinator $S$ and any target set $U\subset\R^{d}$ that is nonempty and is the closure of its interior. 

\begin{figure}
  \centering
        \includegraphics[width=0.6\textwidth]{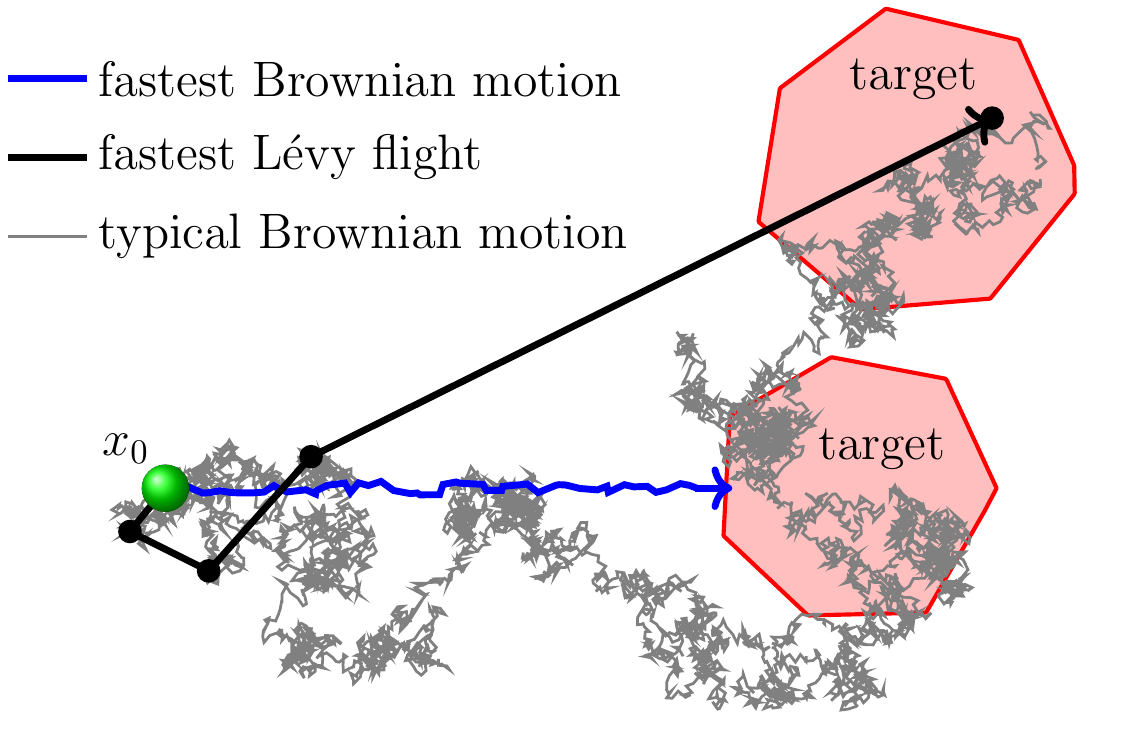}
 \caption{Fastest FHTs of Brownian motions versus L{\'e}vy flights. Starting from $x_{0}$ (green ball), the thick blue path illustrates a Brownian motion that is the first to hit the target (red regions) out of many iid Brownian motions. Such fastest Brownian motions tend to follow the shortest path to the target. The thick black path illustrates the fastest L{\'e}vy flight out of many iid L{\'e}vy flights, which does not take the shortest path to the target. This fastest L{\'e}vy flight is obtained as a random time change of a typical Brownian motion (thin gray path) that wanders around and moves in and out of targets. This illustration of the fastest L{\'e}vy flight is characteristic of any subordinate Brownian motion.}
 \label{figschem2}
\end{figure}

Furthermore, if $\tau_{1},\dots,\tau_{N}$ are independent and identically distributed (iid) realizations of the FHT $\tau$, then we prove that $(\rho N)T_{N}$ converges in distribution to a unit rate exponential random variable as $N$ grows, which means
\begin{align}\label{cd}
\P\Big(T_{N}>\frac{z}{\rho N}\Big)
\to e^{-z}\quad\text{as $N\to\infty$ for each $z\ge0$}.
\end{align}
Hence, $T_{N}$ is approximately exponentially distributed with rate $\rho N$. Furthermore, if $\E[T_{N}]<\infty$ for some $N\ge1$, then we obtain all the moments of $T_{N}$ for large $N$. In particular, we prove that
\begin{align}\label{decay}
\E[T_{N}]
&\sim\sqrt{\text{Variance}[T_{N}]}
\sim\frac{1}{\rho N}\quad\text{as }N\to\infty.
\end{align}
We also extend \eqref{cd} and \eqref{decay} to the $k$th fastest FHT for any $1\le k\ll N$. In the case that $X$ is a superdiffusive L{\'e}vy flight whose probability density satisfies the fractional Fokker-Planck equation in \eqref{ffpe0}, the L{\'e}vy measure $\nu$ used in the rate $\rho$ in \eqref{rho0} is
\begin{align}\label{nu0}
\nu(\dd s)
={{K}}\frac{\alpha/2}{\Gamma(1-\alpha/2)}\frac{1}{s^{1+\alpha/2}}\dd s,\quad s>0.
\end{align}
We emphasize that our results hold for any nondeterministic subordinator $S$ (meaning we exclude only the trivial case that $S(t)$ is a deterministic function $bt$ for some $b\ge0$). In particular, the L{\'e}vy measure $\nu$ of the subordinator $S$ need not have the slow power law decay in \eqref{nu0} which gives rise to long jumps of $X$ in \eqref{tc0}. Examples of other subordinators commonly used in modeling are given in section~\ref{othersubs}.

Before outlining the rest of the paper, we comment on how our results on subordinate Brownian motions relate to extreme statistics and large deviation theory for standard diffusion processes (i.e.\ processes satisfying a standard drift-diffusion It\^{o} stochastic differential equation). First, the $1/N$ decay in \eqref{decay} is much faster than the well-known $1/\ln N$ decay for diffusion processes \cite{weiss1983}. Indeed, the extreme FHT for diffusion processes has the following rather slow decay in mean as the number of searchers $N$ grows \cite{lawley2020uni},
\begin{align}\label{diff0}
\E[T_{N}^{\text{diff}}]
\sim\frac{L^{2}}{4D\ln N}\quad\text{as }N\to\infty,
\end{align}
where $L>0$ is a certain geodesic distance from the possible searcher starting locations to the target and $D>0$ is the characteristic diffusivity. Further contrasting \eqref{decay} and \eqref{diff0}, a salient feature of extreme FHTs of diffusion processes is that they only depend on the shortest path to the target since the fastest searchers follow this geodesic path \cite{lawley2020uni}. In particular, extreme FHTs of diffusion are unaffected by changes to the problem outside of this path, such as altering the size of the target, the size of the domain, or even the space dimension $d\ge1$. In contrast, it is evident from the formula for the rate $\rho$ in \eqref{rho0} that the extreme FHTs of subordinate Brownian motion depend on all these global properties of the problem, which reflects the fact that the fastest subordinate Brownian searchers do not take a direct path to the closest part of the target. These results are illustrated in Figure~\ref{figschem2} for the case of a L{\'e}vy flight (though the illustration is characteristic of any subordinate Brownian motion).

These differences stems from the difference between our result in \eqref{st0} for subordinate Brownian motion and Varadhan's formula for diffusion processes \cite{varadhan1967}. Varadhan's formula is a celebrated result in large deviation theory which implies that if $X^{\text{diff}}=\{X^{\text{diff}}(t)\}_{t\ge0}$ is a diffusion process, then
\begin{align}\label{varadhan0}
\lim_{t\to0+}t\ln\P(X^{\text{diff}}(t)\in U)
=-L^{2}/(4D)<0,
\end{align}
where $L$ and $D$ are as in \eqref{diff0}. The result in \eqref{st0} can thus be interpreted as a type of Varadhan's formula for subordinate Brownian motion. 

The rest of the paper is organized as follows. In section~\ref{prelim}, we review some definitions and results from probability theory. In section~\ref{math}, we present our general mathematical results. In section~\ref{examples}, we illustrate our results in several examples and compare the theory to numerical simulations. We conclude by discussing relations to prior work. Proofs are presented in the appendix.

\section{Preliminaries}\label{prelim}

We begin by reviewing properties of subordinators, subordinate Brownian motions, L{\'e}vy flights, fractional Laplacians, and related concepts.

\subsection{Subordinators}\label{sub}

A L{\'e}vy process is a continuous-time stochastic process that has iid increments and satisfies certain technical conditions \cite{bertoin1996}. 
A subordinator is a one-dimensional, nondecreasing L{\'e}vy process $S=\{S(t)\}_{t\ge0}$ with $S(0)=0$. The distribution of $S$ is determined by its Laplace exponent ${{\Phi}}(\beta)$, which satisfies
\begin{align}
\E\big[e^{-\beta S(t)}\big]
&=e^{-t{{\Phi}}(\beta)},\quad\text{for all $t\ge0$ and $\beta\ge0$},\nonumber
\\
{{\Phi}}(\beta)
&=b\beta
+\int_{0}^{\infty}(1-e^{-\beta s})\,\nu(\dd s),\quad\text{for all $\beta\ge0$},\label{le}
\end{align}
where $b\ge0$ is the drift and $\nu$ is the L{\'e}vy measure. In particular, $\nu$ satisfies 
\begin{align*}
\nu((-\infty,0])=0
\quad\text{and}\quad
\int_{0}^{\infty}\min\{1,s\}\,\nu(\dd s)
<\infty.
\end{align*}
A L{\'e}vy measure $\nu(\dd s)$ can be interpreted as the rate that $S$ increases by $s>0$. 

A subordinator $S$ is called an $(\alpha/2)$-stable subordinator for $\alpha\in(0,2)$ if it satisfies the following self-similarity or scaling property,
\begin{align}\label{sss}
t^{-2/\alpha}S(t)
=_{\dd}S(1)\quad \text{for all }t>0,
\end{align}
where $=_{\dd}$ denotes equality in distribution. In this case, $S$ is a pure jump process (i.e. zero drift $b=0$) with Laplace exponent $\Phi(\beta)=K\beta^{\alpha/2}$ and L{\'e}vy measure in \eqref{nu0} for some $K>0$. Examples of other subordinators are given in section~\ref{othersubs}.

\subsection{Subordinate Brownian motion}

For any dimension $d\ge1$, let ${{B}}=\{{{B}}(s)\}_{s\ge0}$ be a $d$-dimensional Brownian motion with mean-squared displacement \begin{align}\label{msdw}
\E\big[\|{{B}}(s)\|^{2}\big]
=2ds\quad\text{for all $s\ge0$}.
\end{align}
It is well-known that ${{B}}$ satisfies the diffusive scaling property,
\begin{align}\label{brownianscaling}
s^{-1/2}{{B}}(s)=_{\dd}{{B}}(1)\quad\text{for all }s>0.
\end{align}
If $S$ is an independent subordinator with Laplace exponent ${{\Phi}}$, then the L{\'e}vy process $X=\{X(t)\}_{t\ge0}$ defined by
\begin{align}\label{tc}
X(t)
:={{B}}(S(t))+X(0),\quad t\ge0,
\end{align}
is called a subordinate Brownian motion \cite{kim2012}. That is, $X$ is a random time change of Brownian motion. We assume that the possibly random initial condition $X(0)\in\R^{d}$ is independent of ${{B}}$ and $S$. The L{\'e}vy exponent of $X$ is $\Phi(|\xi|^{2})$, meaning
\begin{align*}
\E\big[e^{i\xi\cdot(X(t)-X(0))}\big]
=e^{-t\Phi(|\xi|^{2})},\quad\xi\in\R^{d},\,t\ge0.
\end{align*}
Subordinate Brownian motions are said to be isotropic since their L{\'e}vy exponent depends only on $|\xi|^{2}$. The infinitesimal generator of $X$ can be written as $-\Phi(-\Delta)$, where $\Delta$ is the Laplacian in $\R^{d}$ \cite{kim2012}. It follows immediately from \eqref{msdw}-\eqref{tc} that the mean-squared displacement of $X$ is
\begin{align*}
\E\big[\|X(t)-X(0)\|^{2}\big]
=2d\E[S(t)]\quad\text{for all }t\ge0.
\end{align*}

\subsection{L{\'e}vy flights}

If $S$ is an $(\alpha/2)$-stable subordinator with $\alpha\in(0,2)$ as in \eqref{nu0}, then we call the corresponding subordinate Brownian motion ${{X}}=\{{{X}}(t)\}_{t\ge0}$ in \eqref{tc} a L{\'e}vy flight \cite{dubkov2008}. It follows immediately from \eqref{sss}-\eqref{brownianscaling} that a L{\'e}vy flight ${{X}}$ satisfies the superdiffusive scaling property,
\begin{align}\label{scaling}
t^{-1/\alpha}{{X}}(t)
=_{\dd}{{X}}(1)\quad \text{for all }t>0.
\end{align}
L{\'e}vy flights arise as a scaling limit of a random walk with heavy-tailed, power law jumps \cite{metzler2004}. The probability density function for the position of the L{\'e}vy flight satisfies the space fractional Fokker-Planck equation in \eqref{ffpe0} \cite{meerschaert2019}. 

\subsection{First hitting times (FHTs)}

Let $\tau>0$ denote the FHT of the subordinate Brownian motion $X$ in \eqref{tc} to some target set $U\subset\R^{d}$,
\begin{align}\label{tau1}
\tau
:=\inf\{t>0:X(t)\in U\},
\end{align}
and let $\sigma>0$ denote the FHT of the Brownian motion ${{B}}$ to $U$,
\begin{align}\label{sigma}
\sigma
:=\inf\{s>0:{{B}}(s)\in U\}.
\end{align}
We are not interested in the behavior of $X$ after time $\tau$, and thus it is enough to consider the so-called stopped subordinate Brownian motion,
\begin{align}\label{stopsub}
X(\min\{\tau,t\})={{B}}(S(\min\{\tau,t\})).
\end{align}
In \eqref{stopsub}, we first subordinate Brownian motion and then stop the process when it hits the target. Reversing the order of these two operations gives the so-called subordinate stopped Brownian motion, 
\begin{align}\label{substop}
\widetilde{X}(t)
:={{B}}(\min\{\sigma,S(t)\})\quad t\ge0.
\end{align}
The FHT of \eqref{substop} to $U$ is,
\begin{align}\label{taut}
\widetilde{\tau}
:=\inf\{t>0:\widetilde{X}(t)\in U\}
=\inf\{t>0:S(t)>\sigma\}.
\end{align}
While we are primarily interested in $\tau$ in \eqref{tau1} rather than $\widetilde{\tau}$ in \eqref{taut}, the fact that $\widetilde{\tau}\le\tau$ almost surely plays an important role in studying $\tau$. 

\section{General analysis}\label{math}

In this section, we present our general analysis and results on subordinate Brownian motions. We begin with two propositions.

\subsection{Two useful propositions}

The first proposition computes the generator of a subordinator in a case that is useful for our analysis. 

\begin{proposition}\label{ias2}
Assume $F:[0,\infty)\to[0,1]$ is Lipschitz continuous and satisfies \begin{align}\label{fv}
F(0)=0\quad\text{and}\quad
F'(0)
:=\lim_{s\to0+}\frac{F(s)}{s}\in[0,\infty).
\end{align}
If $S=\{S(t)\}_{t\ge0}$ is a subordinator with drift $b\ge0$ and L{\'e}vy measure $\nu$, then 
\begin{align}\label{fc}
\lim_{t\to0+}\frac{\E[F(S(t))]}{t}
=\rho
:=bF'(0)+\int_{0}^{\infty}F(s)\,\nu(\dd s)
<\infty.
\end{align}
\end{proposition}

Proposition~\ref{ias2} is useful for finding the short-time distribution of functionals of subordinated processes. The next proposition shows how the short-time distribution of a single FHT yields the asymptotic behavior of extreme FHTs.

Before stating the proposition, we recall a few definitions. A random variable $T$ has an exponential distribution with rate $\lambda>0$ if $\P(T\le t)=1-e^{-\lambda t}$ for $t\ge0$. If $\{T_{i}\}_{i=1}^{k}$ are $k\ge1$ iid exponential random variables with rate $\lambda>0$, then their sum has an Erlang distribution with rate $\lambda>0$ and shape $k\in\{1,2,3,\dots\}$, which means
\begin{align*}
\P\Big(\sum_{i=1}^{k}T_{i}\le t\Big)
=1-\frac{\Gamma(k,\lambda t)}{\Gamma(k)},\quad t\ge0,
\end{align*}
where $\Gamma(a,z):=\int_{z}^{\infty}u^{a-1}e^{-u}\,\dd u$ is the upper incomplete gamma function. A sequence of random variables $\{Z_{N}\}_{N\ge1}$ converges in distribution to  $Z$ as $N\to\infty$ if 
\begin{align*}
\P(Z_{N}\le z)
\to\P(Z\le z)\quad\text{as }N\to\infty,
\end{align*}
for all points $z\in\R$ such that $F(z):=\P(Z\le z)$ is continuous. If $\{Z_{N}\}_{N\ge1}$ converges in distribution to an Erlang random variable with rate $\lambda$ and shape $k$, then we write $Z_{N}\to_{\dd}\textup{Erlang}(\lambda,k)$, and if $k=1$, then we write $Z_{N}\to_{\dd}\textup{Exponential}(\lambda)$.

\begin{proposition}\label{maink}
Let $\{\tau_{n}\}_{n\ge1}$ be an iid sequence of random variables with
\begin{align}\label{shortk}
\P(\tau_{n}\le t)
&\sim{{\rho}} t\quad\text{as }t\to0+,
\end{align}
for some rate $\rho>0$. Let $T_{k,N}$ be the $k$th order statistic,
\begin{align}\label{tkn}
T_{k,N}
:=\min\big\{\{\tau_{1},\dots,\tau_{N}\}\backslash\cup_{j=1}^{k-1}\{T_{j,N}\}\big\},\quad k\in\{1,\dots,N\},
\end{align}
where $T_{1,N}:=\min\{\tau_{1},\dots,\tau_{N}\}$. The following rescaling of $T_{k,N}$ converges in distribution to an Erlang random variable with unit rate and shape $k$,
\begin{align*}
(\rho N)T_{k,N}
\to_{\dd}
\textup{Erlang}(1,k)\quad\text{as }N\to\infty.
\end{align*}
If we assume further that $\E[T_{1,N}]<\infty$ for some $N\ge1$, then 
\begin{align*}
\E[(T_{k,N})^{m}]
&\sim
\frac{\Gamma(k+m)}{\Gamma(k)}\frac{1}{(\rho N)^{m}}
\quad\text{ for each moment $m\in(0,\infty)$ as $N\to\infty$}.
\end{align*}
\end{proposition}

Proposition~\ref{maink} is a special case of Theorems 5 and 6 in \cite{madrid2020comp} which were proven for the case $\P(\tau\le t)\sim \rho t^{q}$ as $t\to0$ for some $\rho>0$ and $q>0$.

\subsection{Subordinated processes}\label{gs}

Before considering subordinate Brownian motion, we first analyze subordinate processes when the ``parent'' process is not necessarily Brownian. Let $S=\{S(t)\}_{t\ge0}$ be a subordinator with drift $b\ge0$ and L{\'e}vy measure $\nu$ as in section~\ref{sub}. Let $Y=\{Y(s)\}_{s\ge0}$ be a stochastic process independent of $S$. Define the FHT to a set $U$ in the state space of $Y$,
\begin{align*}
\sigma
:=\inf\{s>0:Y(s)\in U\}.
\end{align*}
Define the two subordinations of the ``parent'' process $Y$,
\begin{align*}
X(t)
&:=Y(S(t)),\quad
\widetilde{X}(t)
:=Y(\min\{\sigma,S(t)\}),\quad t\ge0.
\end{align*}
Define the FHTs of $X$ and $\widetilde{X}$ to $U$,
\begin{align*}
\ot
&:=\inf\{t>0:X(t)\in U\},\quad
\t
:=\inf\{t>0:\widetilde{X}(t)\in U\}
=\inf\{t>0:S(t)>\sigma\}.
\end{align*}

Since $Y$ and $S$ are independent, conditioning on the value of $S(t)$ gives
\begin{align*}
\P(\t\le t)
=\P(\sigma\le S(t))
=\E[\widetilde{F}(S(t))],\quad t\ge0,
\end{align*}
where $\widetilde{F}(s):=\P(\sigma\le s)$. Therefore, if $\widetilde{F}(s)$ is merely Lipschitz and satisfies \eqref{fv}, then Proposition~\ref{ias2} yields the short-time behavior of the distribution of $\t$,
\begin{align}\label{abv}
\lim_{t\to0+}\frac{\P(\t\le t)}{t}
=\tilde{\rho}
:=b\widetilde{F}'(0)+\int_{0}^{\infty}\widetilde{F}(s)\,\nu(\dd s)
<\infty.
\end{align}
Furthermore, if $\tilde{\rho}>0$ and $\widetilde{T}_{k,N}$ is the $k$th fastest FHT of $N$ iid realizations of $\t$ (see \eqref{tkn}), then Proposition~\ref{maink} yields the large $N$ distribution of $\widetilde{T}_{k,N}$ in terms of an Erlang random variable. Furthermore, if $\E[\widetilde{T}_{1,N}]<\infty$ for some $N\ge1$, then Proposition~\ref{maink} also yields the large $N$ behavior of the $m$th moment of $\widetilde{T}_{k,N}$.

Next, notice that we have the following bounds on the distribution of the FHT $\tau$,
\begin{align}\label{gb}
\P(X(t)\in U)
\le\P(\ot\le t)
\le\P(\t\le t)\quad\text{for all }t\ge0,
\end{align}
since $\t\le\ot$ almost surely and $X(t)\in U$ implies $\ot\le t$. Since $Y$ and $S$ are independent, we again condition on the value of $S(t)$ to obtain
\begin{align*}
\P(X(t)\in U)
=\E[F(S(t))],\quad t\ge0,
\end{align*}
where $F(s):=\P(Y(s)\in U)$. Therefore, if $F(s)$ is Lipschitz and satisfies \eqref{fv}, then Proposition~\ref{ias2} yields
\begin{align}\label{abv2}
\lim_{t\to0+}\frac{\P(X(t)\in U)}{t}
=\rho
:=bF'(0)+\int_{0}^{\infty}F(s)\,\nu(\dd s)
<\infty.
\end{align}
Therefore, the bounds in \eqref{gb} and the limits in \eqref{abv} and \eqref{abv2} yield the following bounds on the short-time behavior of the distribution of $\tau$,
\begin{align*}
\rho t+o(t)
\le\P(\tau\le t)
\le\tilde{\rho}t+o(t)\quad\text{as }t\to0+,
\end{align*}
where $f(t)=o(t)$ means $f(t)/t\to0$. 
If $T_{k,N}$ is the $k$th fastest FHT of $N$ iid realizations of $\tau$ (see \eqref{tkn}), $\rho\tilde{\rho}>0$, and $\E[T_{1,N}]<\infty$ for some $N\ge1$, then it follows from Proposition~\ref{maink} that we can bound the decay of the $m$th moment of $T_{k,N}$ as $N\to\infty$,
\begin{align*}
\frac{\Gamma(k+m)}{\Gamma(k)}\frac{1}{(\tilde{\rho} N)^{m}}+o(N^{-m})
\le\E[(T_{k,N})^{m}]
\le\frac{\Gamma(k+m)}{\Gamma(k)}\frac{1}{(\rho N)^{m}}+o(N^{-m}).
\end{align*}

Summarizing, if $X$ is defined by subordinating some process $Y$, then Proposition~\ref{ias2} yields information about the short-time distribution of $X$ and FHTs of $X$. Then, Proposition~\ref{maink} translates this short-time distribution of a single FHT into the behavior of extreme FHTs. Importantly, these conclusions require only mild assumptions on the parent process $Y$. In the next subsection, we consider the case that the parent process is a Brownian motion.

\subsection{Subordinate Brownian motion}\label{sb}

Let $S=\{S(t)\}_{t\ge0}$ be a subordinator as in section~\ref{sub} and assume that $S$ has nontrivial L{\'e}vy measure,
\begin{align}\label{nontrivial}
\nu((0,\infty))>0,
\end{align}
to exclude the trivial case in which $S$ is the deterministic function $S(t)=bt$ for all $t\ge0$. Let ${{B}}=\{{{B}}(s)\}_{s\ge0}$ be an independent, $d$-dimensional Brownian motion for any $d\ge1$ as in \eqref{msdw}. Define $X=\{X(t)\}_{t\ge0}$ as the random time change of ${{B}}$,
\begin{align}\label{tc3}
X(t)
:={{B}}(S(t))+X(0),\quad t\ge0,
\end{align}
where $X(0)\in\R^{d}$ is a possibly random initial position independent of $S$ and ${{B}}$. 

Let $\tau$ be the FHT of $X$ to some target set $U\subset\R^{d}$ (see \eqref{tau1}). Assume $U$ is nonempty and is the closure of its interior, which precludes trivial cases such as the target having zero Lebesgue measure. Assume that the distribution of $X(0)$ is a probability measure with compact support $U_{0}\subset\R^{d}$ that does not intersect the target,
\begin{align}\label{away}
U_{0}\cap U
=\varnothing.
\end{align}
Note that $U_{0}$ and $U$ are both closed sets, and thus \eqref{away} ensures that $U_{0}$ and $U$ are separated by a strictly positive distance. As two examples, the initial distribution could be a Dirac mass at a point $X(0) = x_{0} = U_{0}\in\R^{d}$ if $x_{0}\notin U$ or it could be uniform on a set $U_{0}$ satisfying \eqref{away}. 

\begin{theorem}\label{key}
Under the assumptions of section~\ref{sb}, we have that
\begin{align}
\P({\os}\le t)
&\sim \P(X(t)\in U)
\sim \rho t\quad\text{as }t\to0+,\label{tae}\\
\text{where}\qquad
\rho
&:=\int_{0}^{\infty}\P({{B}}(s)+X(0)\in U)\,\nu(\dd s)\in(0,\infty).\label{rhotheorem}
\end{align}
Furthermore, if $T_{N}:=\min\{\tau_{1},\dots,\tau_{N}\}$, where $\{\tau_{n}\}_{n\ge1}$ is an iid sequence of realizations of $\tau$, then
\begin{align}\label{ted}
(\rho N)T_{N}
\to_{\dd}
\textup{Exponential}(1)\quad\text{as }N\to\infty.
\end{align}
More generally, if $T_{k,N}$ is the $k$th fastest FHT in \eqref{tkn}, then
\begin{align}\label{tedk}
(\rho N)T_{k,N}
\to_{\dd}
\textup{Erlang}(1,k)\quad\text{as }N\to\infty.
\end{align}
If $\E[T_{N}]<\infty$ for some $N\ge1$, then
\begin{align}\label{tm}
\E[(T_{k,N})^{m}]
&\sim
\frac{\Gamma(k+m)}{\Gamma(k)}\frac{1}{(\rho N)^{m}}
\quad\text{for each moment $m\in(0,\infty)$ as $N\to\infty$}.
\end{align}
\end{theorem}

Before applying Theorem~\ref{key} to some examples in section~\ref{examples}, we make several comments. First, the asymptotic equality $\P({\os}\le t)\sim \P(X(t)\in U)$ in \eqref{tae} means that paths which hit the target before a short time $t$ are much more likely to stay in the target than to leave before $t$. While this is intuitive, it does not hold for Brownian motion, except on a logarithmic scale (the assumption in \eqref{nontrivial} means that $X$ cannot be a Brownian motion). Second, \eqref{ted} means that $T_{N}$ is approximately exponentially distributed with rate $\rho N$ is $N$ is large, and similarly $T_{k,N}$ is approximately Erlang distributed with rate $\rho N$ and shape $k$. Third, the asymptotics in \eqref{tae} and \eqref{tm} differ markedly from the case of diffusion. Further, the exponential distribution in \eqref{ted} differs from the typically Gumbel distributed extreme FHTs of diffusion \cite{lawley2020dist}. See the Introduction section for more on how Theorem~\ref{key} differs from the diffusion case. Finally, while \eqref{tae} gives the short-time distributions, these are equivalent to the ``small noise'' distributions in the case of a L{\'e}vy flight. Indeed, if ${{X}}$ is a L{\'e}vy flight with generalized diffusion coefficient ${{K}}$, then \eqref{tae} implies
\begin{align*}
\P({{X}}(t)\in U)
\sim {{K}}t\int_{0}^{\infty}\P({{B}}(s)+X(0)\in U)\frac{\alpha/2}{\Gamma(1-\frac{\alpha}{2})}\frac{1}{s^{1+\frac{\alpha}{2}}}\,\dd s\quad\text{as }{{K}}\to0+.
\end{align*}

\section{Examples and numerical simulation}\label{examples}

We now apply Theorem~\ref{key} for various choices of the space dimension $d\ge1$, the target $U$, and the subordinator $S$. 

\subsection{Half-line}\label{half}

Consider a one-dimensional L{\'e}vy flight ${{X}}$ in $\R$ that starts at ${{X}}(0)=0$ with $\alpha\in(0,2)$. That is, $X$ is defined in \eqref{tc3} and $S$ is an $(\alpha/2)$-stable subordinator defined in section~\ref{sub}. Suppose the target is $U=(-\infty,-L]$ for some $L>0$. Theorem~\ref{key} implies that $\tau$ has the short-time distribution in \eqref{tae} with rate
\begin{align*}
\rho
&={{K}}\int_{0}^{\infty}\P({{B}}(s)\in U)\frac{\alpha/2}{\Gamma(1-\alpha/2)}\frac{1}{s^{1+\alpha/2}}\,\dd s
=\frac{\Gamma(\alpha)\sin(\alpha\pi/2)}{\pi}\frac{{{K}}}{{{L}}^{\alpha}}\in(0,\infty),
\end{align*}
since $\P({{B}}(s)\in U)
=\P({{B}}(s)\le -L)
=\tfrac{1}{2}[1+\text{erf}(-L/\sqrt{4s})]$ for $s>0$. This result for this example was derived formally in \cite{palyulin2019}. Theorem~\ref{key} further implies the convergence in distribution in \eqref{ted}-\eqref{tedk}. In addition, the Sparre-Anderson theorem \cite{koren2007} implies that $\P(\tau>t)=\mathcal{O}(t^{-1/2})$ as $t\to\infty$ which implies
\begin{align*}
\E[T_{N}]
=\int_{0}^{\infty}\P(T_{N}>t)\,\dd t
=\int_{0}^{\infty}(\P(\tau>t))^{N}\,\dd t<\infty\quad\text{ if $N\ge3$}.
\end{align*}
Hence, Theorem~\ref{key} implies $\E[(T_{N})^{m}]\sim\Gamma(m+1)(\rho N)^{-m}$ as $N\to\infty$ for any $m>0$.

These conclusions of Theorem~\ref{key} about the asymptotic behavior of $T_{N}$ as $N\to\infty$ are illustrated in Figure~\ref{fighalf} using stochastic simulations (simulation details are given in section~\ref{simulation} below). In the top left panel, we plot the empirical probability density of $(\rho N)T_{N}$ obtained from stochastic simulations with $\alpha=1.5$. As implied by Theorem~\ref{key}, $(\rho N)T_{N}$ converges in distribution to a unit rate exponential random variable. In the top right panel, we plot the maximum difference between the empirical distribution of $(\rho N)T_{N}$ and a unit rate exponential random variable,
\begin{align}\label{ks}
\sup_{z\ge0}\big|\P((\rho N)T_{N}> z)-\exp(-z)\big|,
\end{align}
as a function of $N$ for different choices of $\alpha$. The difference \eqref{ks} is the Kolmogorov-Smirnov distance. This plot shows that the convergence of $(\rho N)T_{N}$ to an exponential random variable is faster for small $\alpha$. In the bottom two plots, we plot the absolute errors between the simulations and the theory for the mean and standard deviation,
\begin{align}\label{aer}
\big|\E[T_{N}]-(\rho N)^{-1}\big|,\quad
\big|\sqrt{\textup{Variance}[T_{N}]}-(\rho N)^{-1}\big|,
\end{align}
as functions of $N$ for $\alpha=1.5$ (bottom left panel) and $\alpha=1$ (bottom right panel). As implied by Theorem~\ref{key}, these errors decay faster than $N^{-1}$ as $N$ grows.

\begin{figure}
  \centering
              \includegraphics[width=0.465\textwidth]{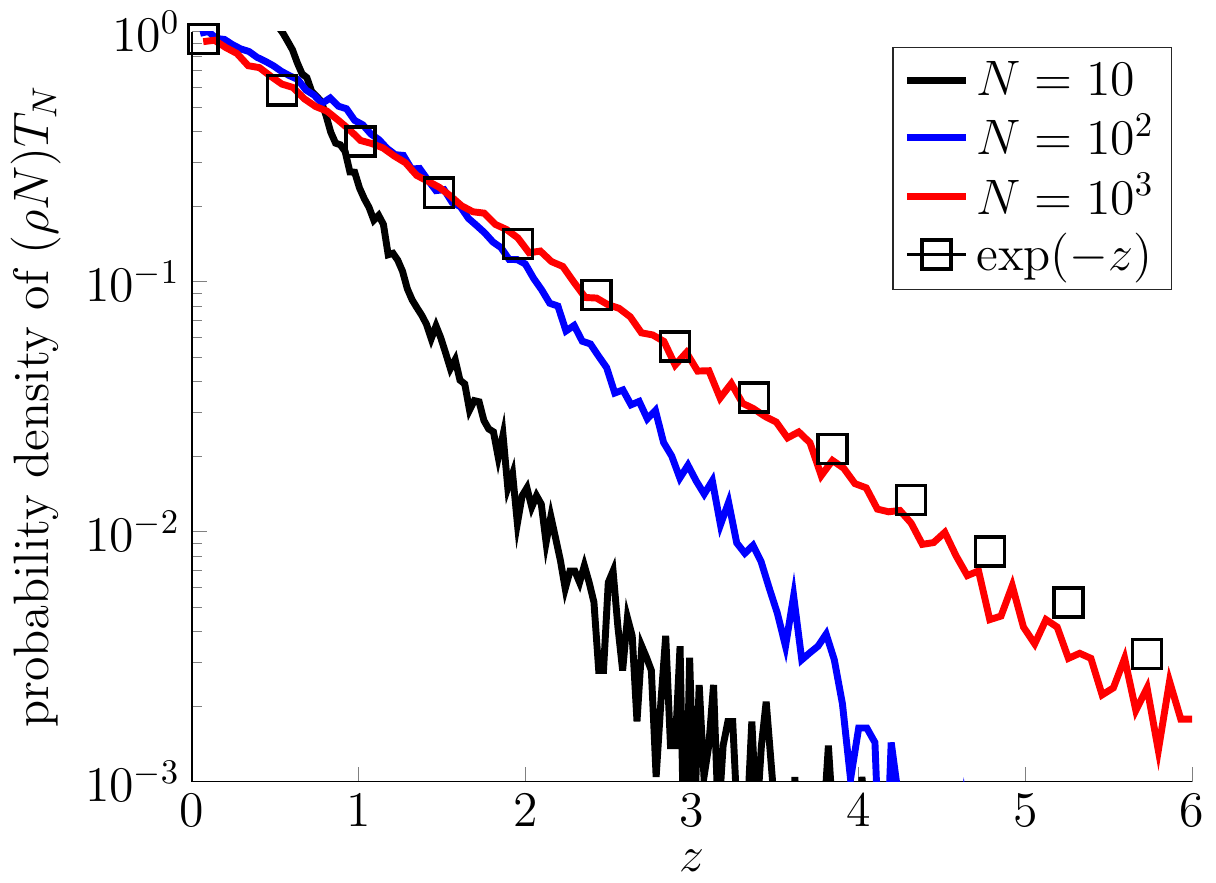}
                \qquad
        \includegraphics[width=0.465\textwidth]{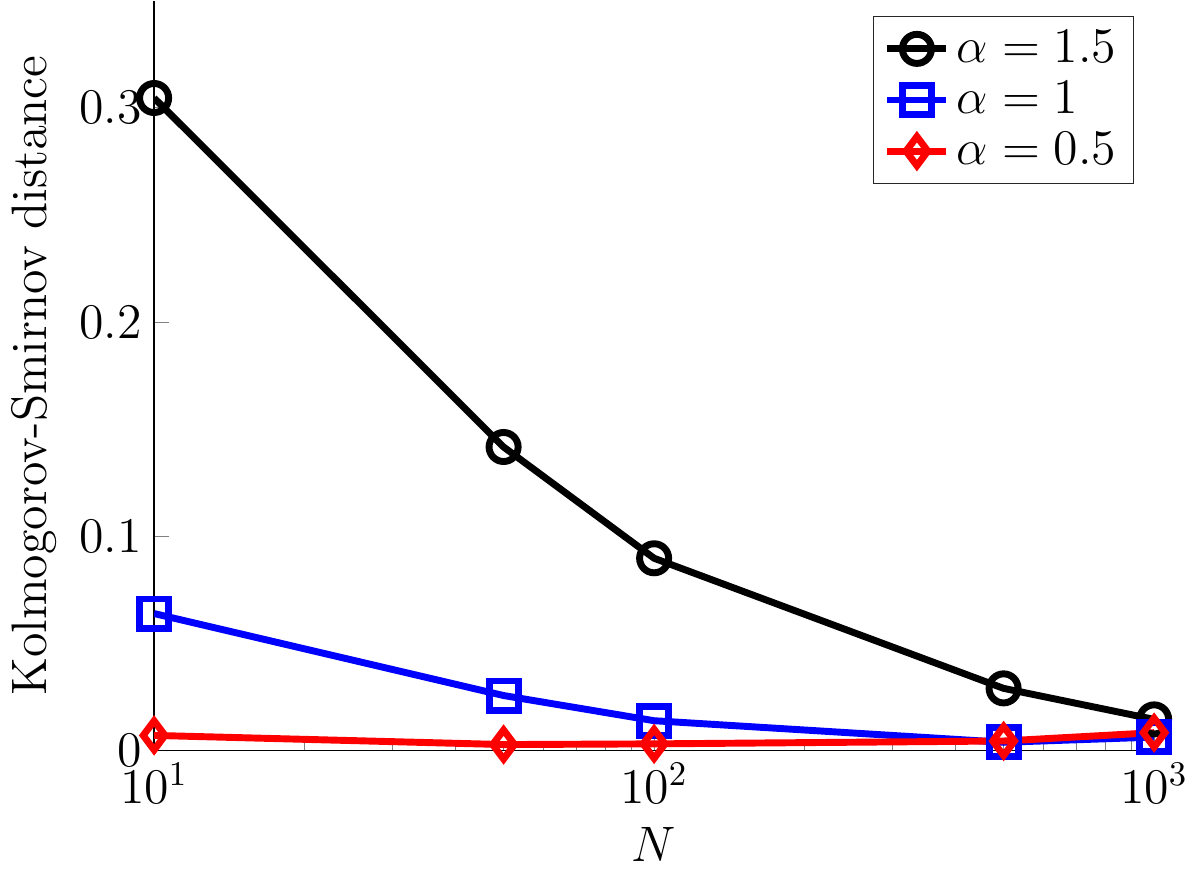}
        \vspace{5pt}
        
    \includegraphics[width=0.465\textwidth]{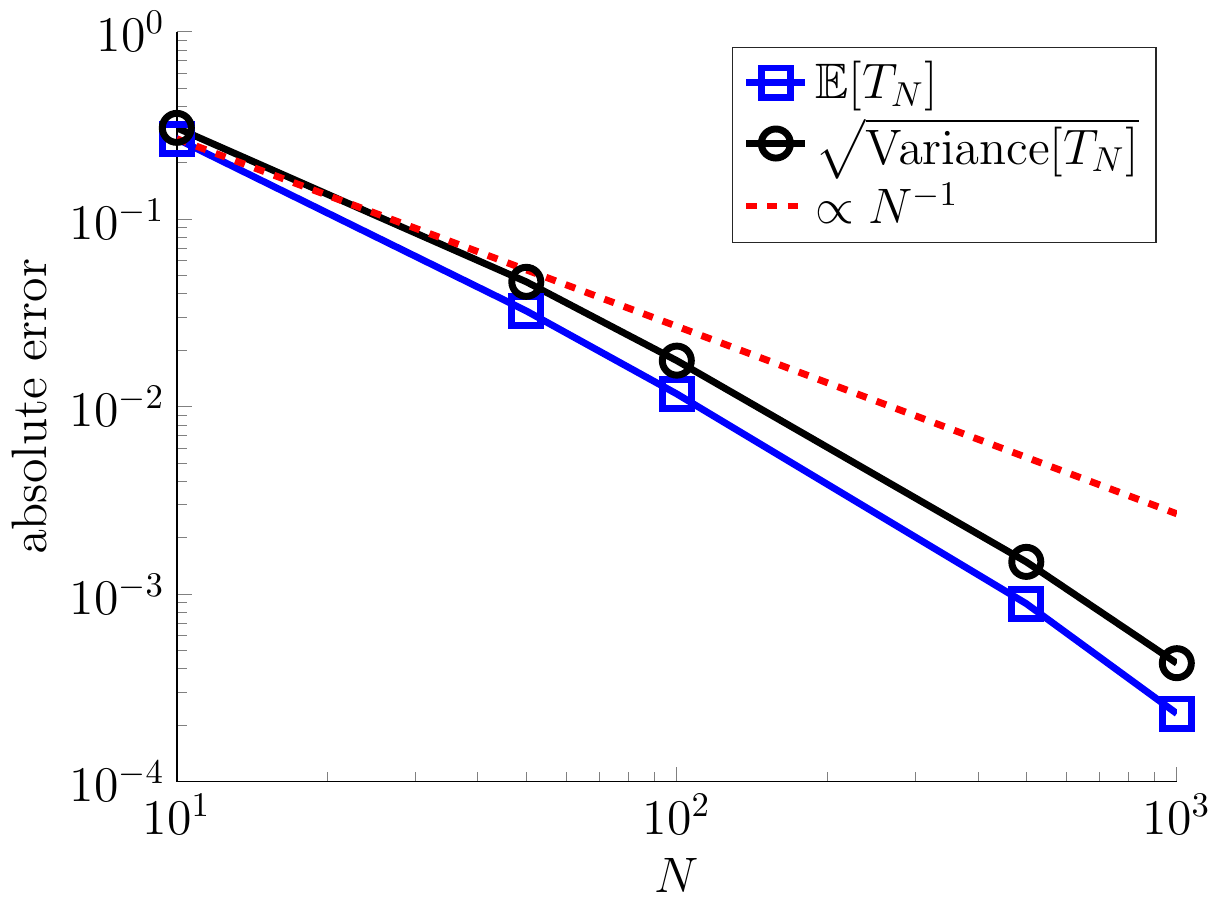}
    \qquad
        \includegraphics[width=0.465\textwidth]{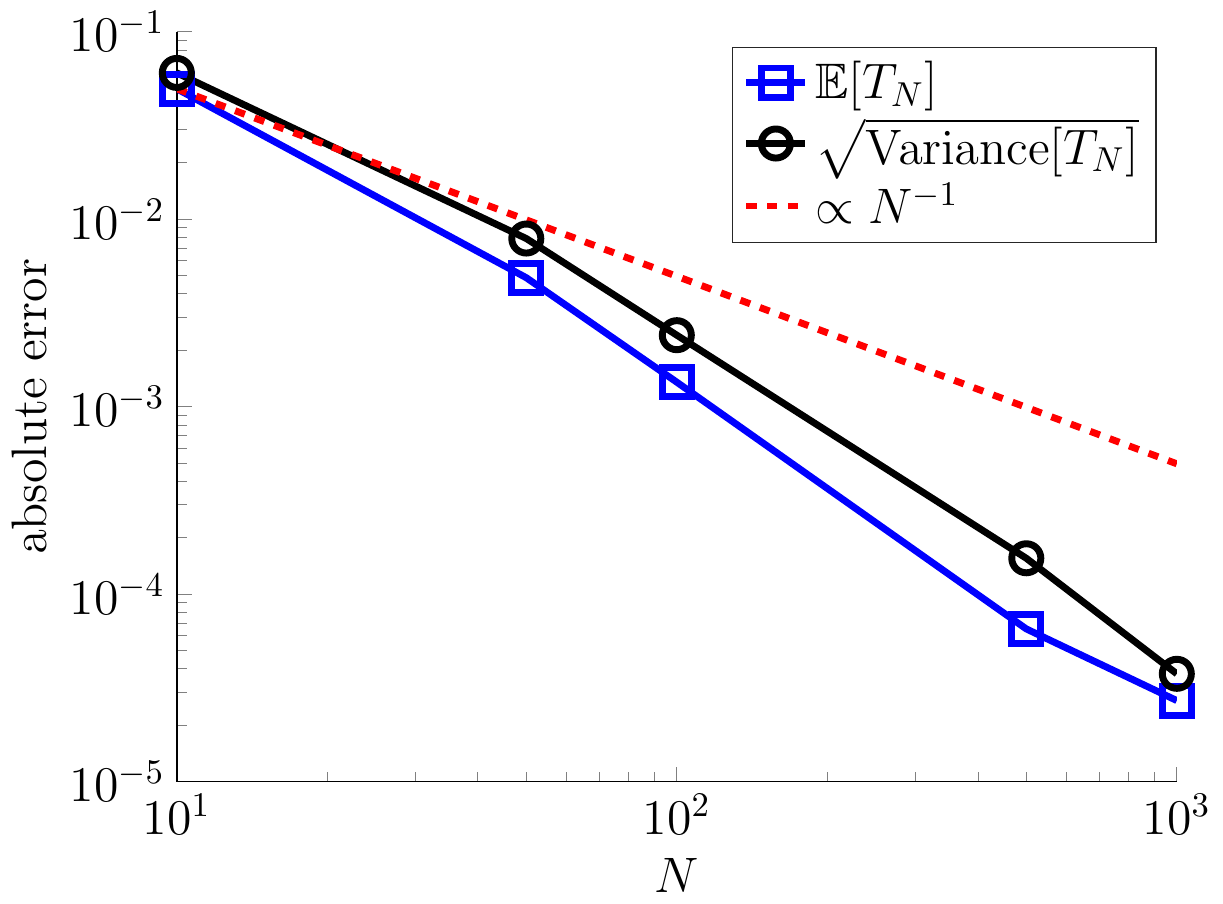}
 \caption{Comparison of Theorem~\ref{key} and empirical results obtained from stochastic simulations for L{\'e}vy flights in the one-dimensional geometry in section~\ref{half}. Top left: Empirical probability density of $(\rho N)T_{N}$ for $\alpha=1.5$. Top right: Kolmogorov-Smirnov distance in \eqref{ks} between the empirical probability density of $(\rho N)T_{N}$ and a unit rate exponential for different choices of $\alpha$. Bottom: Absolute errors for the mean and standard deviation in \eqref{aer} for $\alpha=1.5$ (bottom left) and $\alpha=1$ (bottom right). In all four plots, we take $K=L=1$.}
 \label{fighalf}
\end{figure}

\subsection{Escape from a $d$-dimensional sphere}\label{escape}

Consider a L{\'e}vy flight ${{X}}$ in $\R^{d}$ with $d\ge1$ starting at ${{X}}(0)=0\in\R^{d}$ with $\alpha\in(0,2)$. Suppose the target is
\begin{align}\label{ballt}
U=\{x\in\R^{d}:\|x\|\ge L\},
\end{align}
so that $\tau$ is the escape time from a $d$-dimensional sphere of radius $L>0$ centered at the origin. Theorem~\ref{key} implies that \eqref{tae} holds with
\begin{align}\label{rhol}
\rho
=\rho(L)
&={{K}}\int_{0}^{\infty}\P(\|{{B}}(s)\|\ge L)\frac{\alpha/2}{\Gamma(1-\alpha/2)}\frac{1}{s^{1+\alpha/2}}\,\dd s
=\frac{2^{\alpha}\Gamma(\frac{d+\alpha}{2})}{\Gamma(\frac{d}{2})\Gamma(1-\frac{\alpha}{2})}\frac{{{K}}}{L^{\alpha}},
\end{align}
since $\P(\|{{B}}(s)\|\ge L)
=\Gamma(\frac{d}{2},\frac{L^{2}}{4s})/\Gamma(\frac{d}{2})$ for $s>0$. Theorem~\ref{key} further implies the convergence in distribution in \eqref{ted}-\eqref{tedk}. Furthermore,
\begin{align*}
\E[T_{N}]
\le\E[\tau]
=\Big[\rho\Gamma\Big(1-\frac{\alpha}{2}\Big)\Gamma\Big(1+\frac{\alpha}{2}\Big)\Big]^{-1}<\infty\quad\text{for any }N\ge1,
\end{align*}
where the formula for $\E[\tau]$ is due to Getoor \cite{getoor1961}. Therefore, Theorem~\ref{key} implies that $\E[(T_{N})^{m}]\sim\Gamma(m+1)(\rho N)^{-m}$ as $N\to\infty$ for any moment $m\in(0,\infty)$.

These results are illustrated in Figure~\ref{figball} for dimension $d=3$. In the top left panel, we plot the empirical probability density of $(\rho N)T_{N}$ obtained from stochastic simulations with $\alpha=1.5$, which shows that $(\rho N)T_{N}$ converges in distribution to a unit rate exponential random variable. The top right panel plots the Kolmogorov-Smirnov distance in \eqref{ks} as a function of $N$ for difference choices of $\alpha$. The bottom two plots show the absolute errors for the mean and standard deviation in \eqref{aer} for $\alpha=1.5$ (bottom left panel) and $\alpha=1$ (bottom right panel). As implied by Theorem~\ref{key}, these errors decay faster than $N^{-1}$ as $N$ grows.

\begin{figure}
  \centering
              \includegraphics[width=0.465\textwidth]{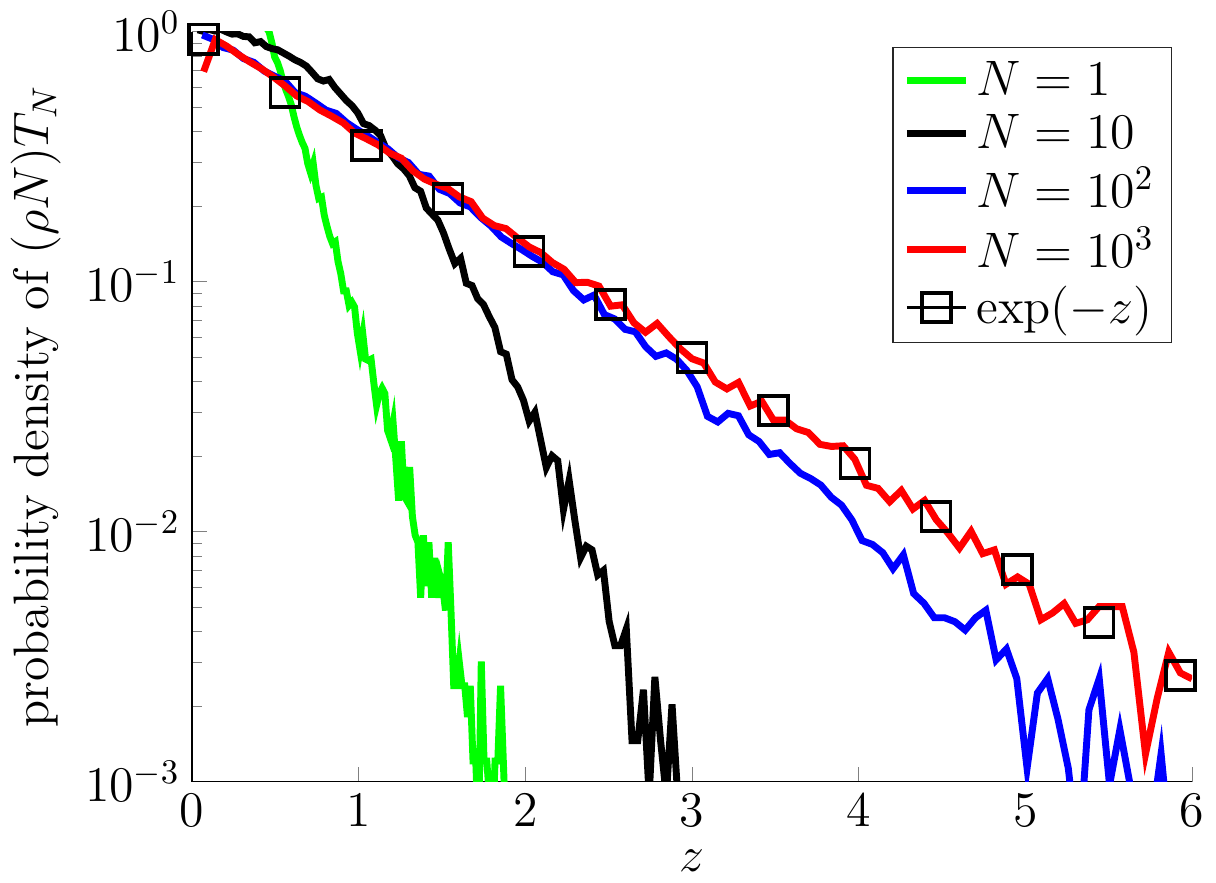}
                \qquad
        \includegraphics[width=0.465\textwidth]{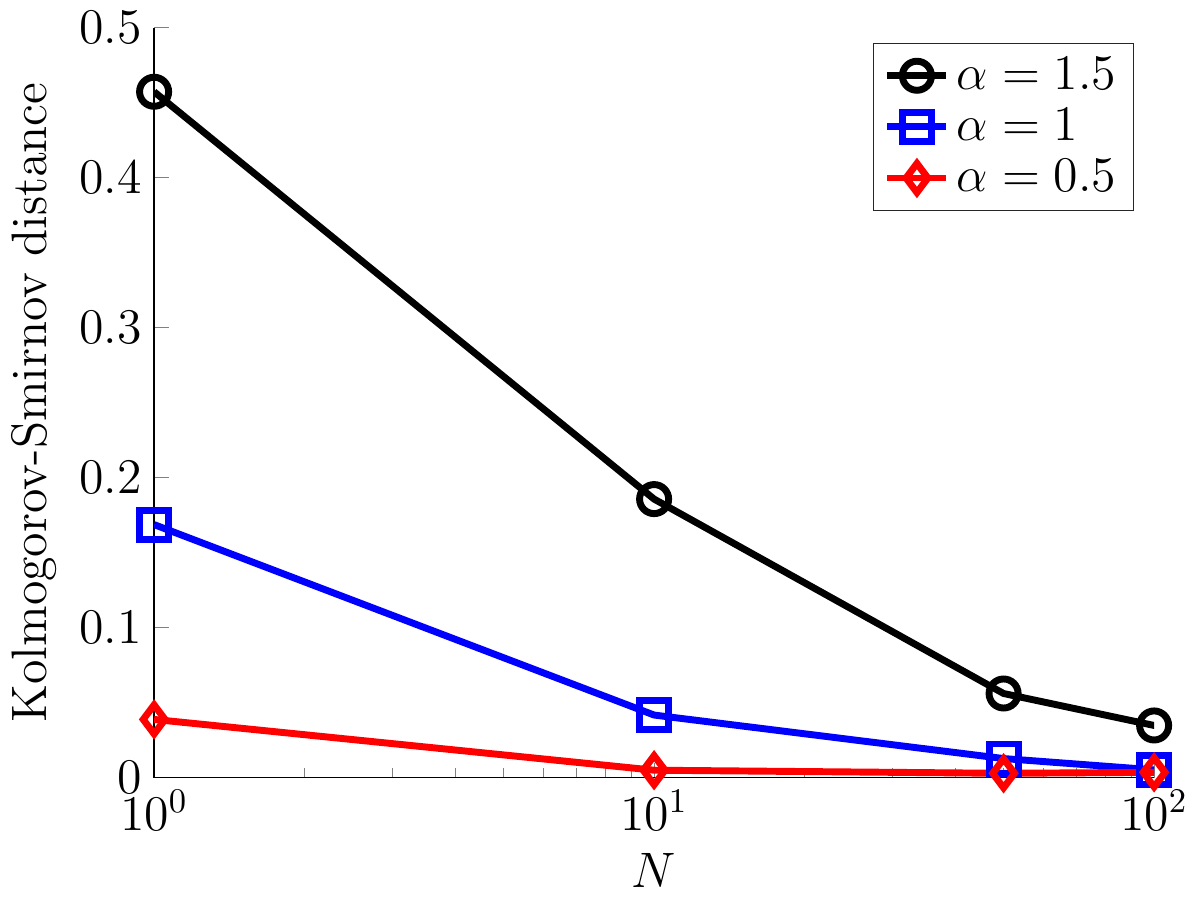}
        \vspace{5pt}
        
            \includegraphics[width=0.465\textwidth]{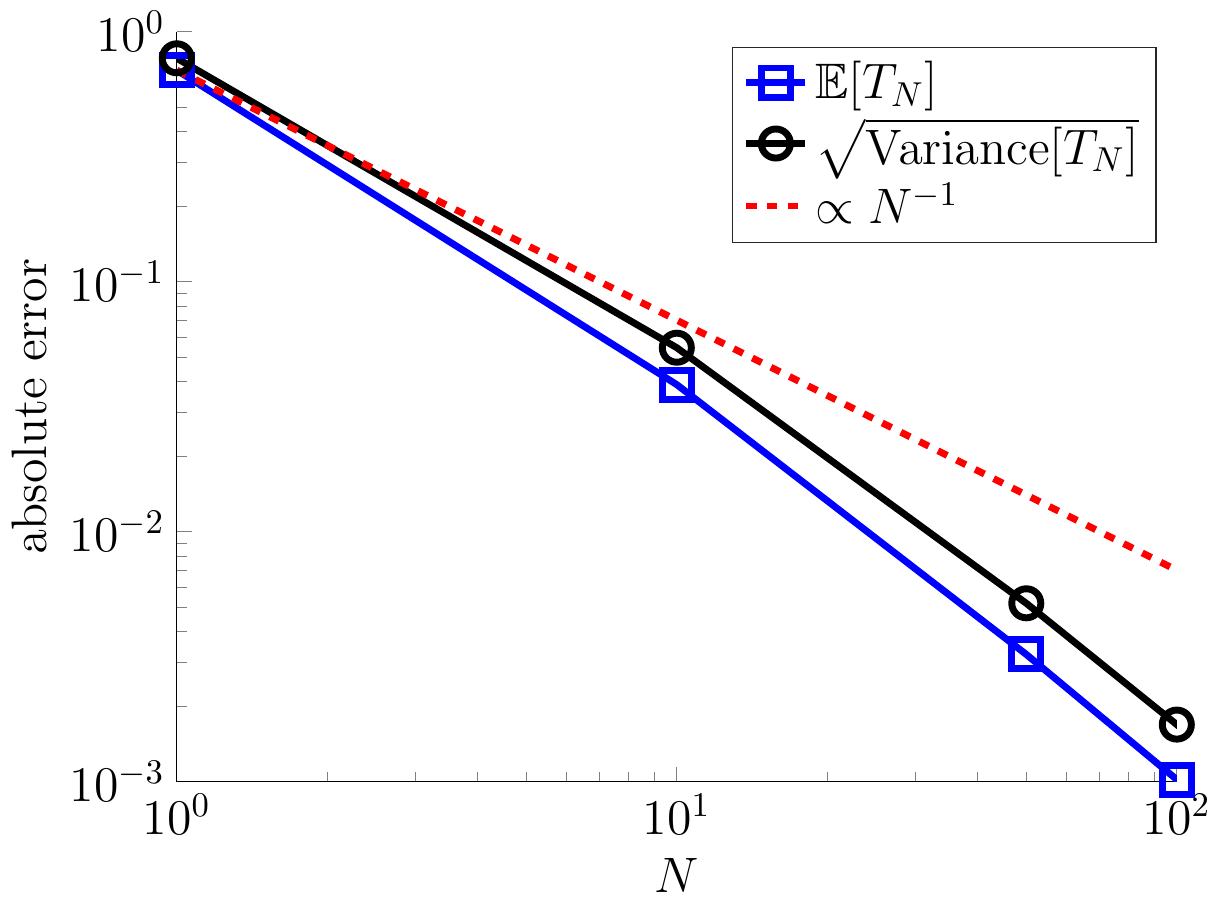}
    \qquad
        \includegraphics[width=0.465\textwidth]{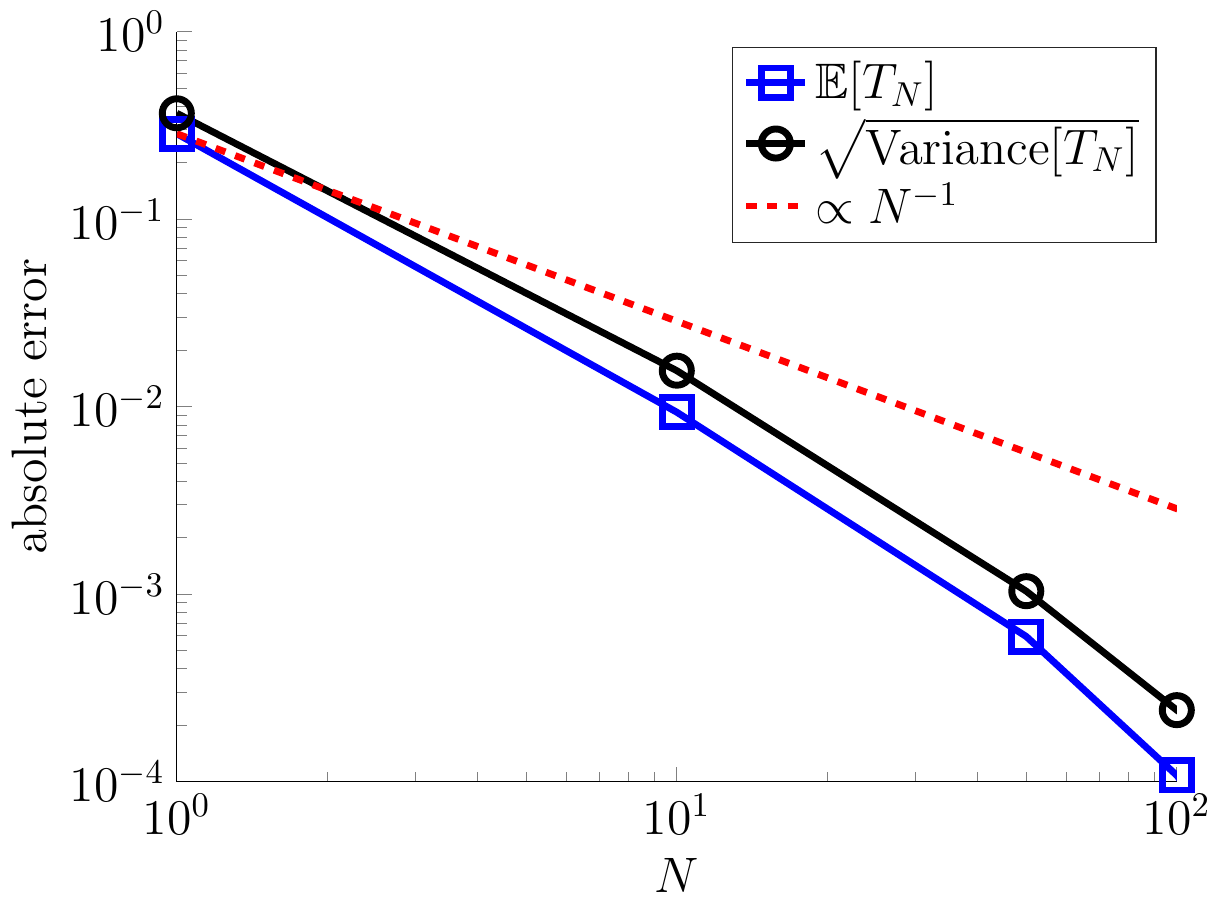}
 \caption{Comparison of Theorem~\ref{key} and empirical results obtained from stochastic simulations for the L{\'e}vy flight escape problem in section~\ref{escape} with $d=3$. Top left: Empirical probability density of $(\rho N)T_{N}$ for $\alpha=1.5$. Top right: Kolmogorov-Smirnov distance in \eqref{ks} between the empirical probability density of $(\rho N)T_{N}$ and a unit rate exponential for different choices of $\alpha$. Bottom: Absolute errors for the mean and standard deviation in \eqref{aer} for $\alpha=1.5$ (bottom left) and $\alpha=1$ (bottom right). In all four plots, we take $K=L=1$.}
 \label{figball}
\end{figure}

We emphasize that the large $N$ decay of the moments of $T_{N}$ for L{\'e}vy flights is much faster than for normal diffusion. To illustrate, let $\tau^{\textup{diff}}$ be the FHT of a pure diffusion process $\{{{X^{\textup{diff}}}}(t)\}_{t\ge0}$ to the target, $\tau^{\textup{diff}}
:=\inf\{t>0:\|{{X^{\textup{diff}}}}(t)\|\ge L\}$. 
The mean FHT is $\E[\tau^{\textup{diff}}]=\frac{L^{2}}{2dD}$ \cite{getoor1961}, where ${{D}}$ is the diffusivity of ${{X^{\textup{diff}}}}$. If $T_{N}^{\textup{diff}}:=\min\{\tau_{1}^{\textup{diff}},\dots,\tau_{N}^{\textup{diff}}\}$ is fastest FHT out of $N$ iid realizations of $\tau^{\textup{diff}}$, then \cite{weiss1983, lawley2020uni}
\begin{align*}
\E[T_{N}^{\textup{diff}}]
\sim\frac{L^{2}}{4{{D}}\ln N}\quad\text{as }N\to\infty.
\end{align*}
Now, it is straightforward to choose the diffusion coefficient of ${{X^{\textup{diff}}}}$ so that $\E[\tau]=\E[\tau^{\textup{diff}}]$. Hence, for these parameters, the mean FHT for a single L{\'e}vy flight and a single diffusion process are identical, but the mean fastest FHT for many L{\'e}vy flights is much faster than for many diffusion processes.  
 
\subsection{Tempered stable subordinator and gamma subordinator}\label{othersubs}

The slow power law decay of the L{\'e}vy measure $\nu$ of the stable subordinator $S$ means that a L{\'e}vy flight $X$ often takes large jumps. This may be undesirable in some modeling situations, and thus it common to ``temper'' the stable subordinator by multiplying its L{\'e}vy measure by a decaying exponential in order to suppress these large jumps. Specifically, the so-called tempered stable subordinator is defined by zero drift and the following Laplace exponent and L{\'e}vy measure,
\begin{align}\label{tempered}
\Phi(\beta)
={{K}}((\beta+\mu)^{\alpha/2}-\mu^{\alpha/2}),\qquad
\frac{\nu(\dd s)}{\dd s}
={{K}}\frac{\alpha/2}{\Gamma(1-\alpha/2)}\frac{e^{-\mu s}}{s^{1+\alpha/2}},\quad s>0,
\end{align}
for $\alpha\in(0,2)$, ${{K}}>0$, and $\mu>0$. Taking $\alpha\to0$ in the exponent in the L{\'e}vy measure of the tempered stable subordinator yields another subordinator commonly used in modeling called the gamma subordinator, which has zero drift and the following Laplace exponent and L{\'e}vy measure for some rate $C>0$,
\begin{align}\label{gammasub}
\Phi(\beta)
={{C}}\log\Big(\frac{\beta+\mu}{\mu}\Big),\qquad
\frac{\nu(\dd s)}{\dd s}
={{C}}\frac{e^{-\mu s}}{s},\quad s>0.
\end{align}

Suppose $S=\{S(t)\}_{t\ge0}$ is the gamma subordinator defined by \eqref{gammasub} and let $X(t):={{B}}(S(t))$ where ${{B}}=\{{{B}}(s)\}_{s\ge0}$ is a 3-dimensional Brownian motion. Letting the target be as in \eqref{ballt}, Theorem~\ref{key} implies that \eqref{tae} holds with
\begin{align*}
\rho
&=C\int_{0}^{\infty}\P(\|{{B}}(s)\|\ge L)\frac{e^{-\mu s}}{s}\,\dd s
=2C\Big(e^{-L\sqrt{\mu}}+\int_{L\sqrt{\mu}}^{\infty}\frac{e^{-z}}{z}\,\dd z\Big).
\end{align*}
Theorem~\ref{key} further implies the convergence in distribution in \eqref{ted}-\eqref{tedk} and the moment behavior in \eqref{tm} (it is straightforward to check that $\E[T_{N}]\le\E[\tau]<\infty$). These results are illustrated in Figure~\ref{figgamma} using stochastic simulations (see section~\ref{simulation}). In the left panel, we illustrate the convergence in distribution in \eqref{ted} by plotting the Kolmogorov-Smirnov distance in \eqref{ks} as a function of $N$. The moment convergence in \eqref{tm} is illustrated in the right panel of Figure~\ref{figgamma}, where we plot the absolute errors for the mean and standard deviation (see \eqref{aer}) as functions of $N$.

\begin{figure}
  \centering
              \includegraphics[width=0.465\textwidth]{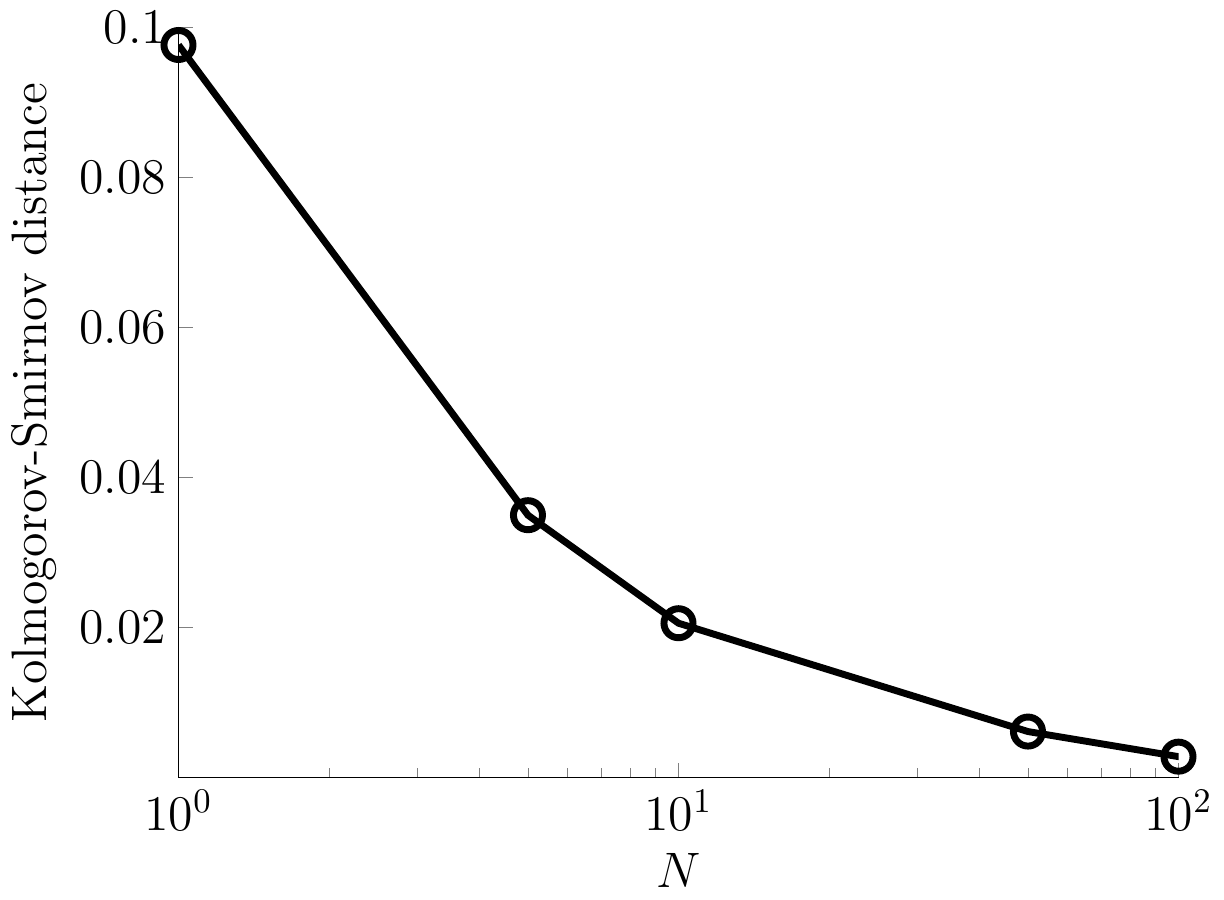}
                \qquad
        \includegraphics[width=0.465\textwidth]{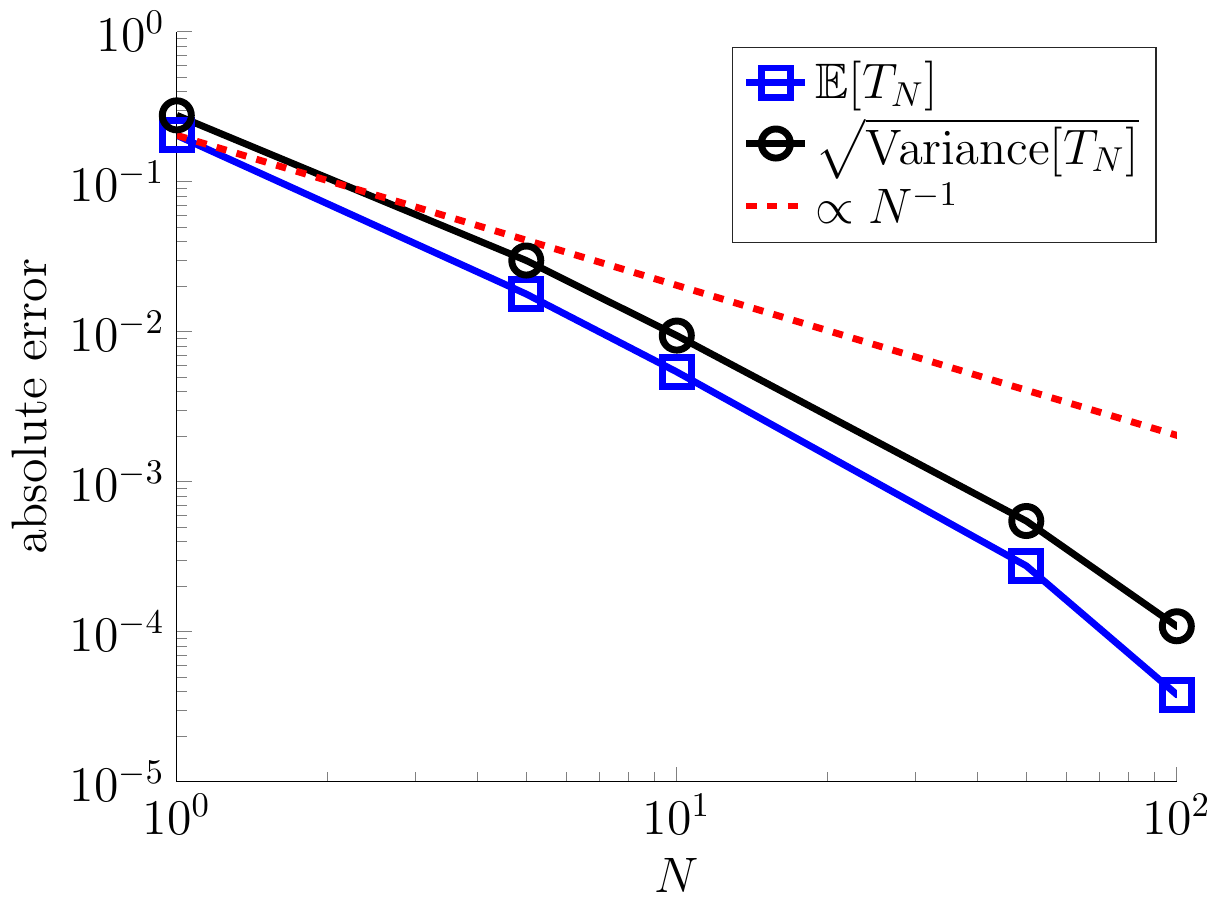}
 \caption{Comparison of Theorem~\ref{key} and empirical results obtained from stochastic simulations for the problem in section~\ref{othersubs} of a Brownian motion subordinated by a gamma subordinator. Left: Kolmogorov-Smirnov distance in \eqref{ks} between the empirical probability density of $(\rho N)T_{N}$ and a unit rate exponential for different choices of $\alpha$. Right: Absolute errors for the mean and standard deviation in \eqref{aer}. In both plots, we take $\mu=C=L=1$.}
 \label{figgamma}
\end{figure}

\subsection{Annular target in $\R^{d}$}

As in section~\ref{escape}, consider a L{\'e}vy flight ${{X}}$ in $\R^{d}$ with ${{X}}(0)=0\in\R^{d}$. However, now suppose that the target is the annular region, 
\begin{align*}
U=\{x\in\R^{d}:0<L_{-}\le\|x\|\le L_{+}\},\quad\text{where }0<L_{-}< L_{+}.
\end{align*}
Hence, \eqref{tae} holds with $\rho=\rho(L_{+})-\rho(L_{-})>0$, where $\rho(L_{\pm})$ is defined in \eqref{rhol} since
\begin{align*}
\P({{B}}(s)\in U)
=\P(\|{{B}}(s)\|\ge L_{-})
-\P(\|{{B}}(s)\|\ge L_{+}).
\end{align*}

This example illustrates some features not seen in the examples above. First, the FHT $\tau$ to $U$ is not the same as the first passage time, $\tau_{\textup{fpt}}:=\inf\{t>0:\|{{X}}(t)\|\ge L_{-}\}$. This is because, in contrast to normal diffusion, ${{X}}$ is a jump process, and therefore it may ``leapover'' the annulus $U$ so that $\tau_{\textup{fpt}}<\tau$. Second, the FHT is infinite with positive probability in dimensions $d\ge3$. That is, there exists $q(d)>0$ so that
\begin{align}\label{qd}
\P(\tau=\infty)
=q(d)>0\quad\text{in dimension }d\ge3.
\end{align}
To see why \eqref{qd} holds, note that ${{X}}$ may leap over $U$ with positive probability. After leaping over $U$, the process starts at some radius larger than $L_{+}$ and may never return to a radius less than $L_{+}$, as a result of the strong Markov property and the fact that Brownian motion is transient if $d\ge3$. Third, \eqref{qd} implies that $\P(T_{N}=\infty)
=(q(d))^{N}>0$ if $d\ge3$. Therefore, the mean fastest FPT is infinite if $d\ge3$,
\begin{align*}
\E[T_{N}]=\infty\quad\text{for every $N\ge1$ if $d\ge3$}.
\end{align*}
Hence, Theorem~\ref{key} ensures that the convergence in distribution in \eqref{ted}-\eqref{tedk} holds, but the moment asymptotics in \eqref{tm} do not hold.

\subsection{Poisson distributed targets in $\R^{d}$}

Consider again a L{\'e}vy flight ${{X}}$ in $\R^{d}$. Studies of the efficiency of superdiffusive search often consider Poisson distributed targets \cite{levernier2020}. To illustrate, suppose $\{x_{i}\}_{i\in\mathbb{N}}$ is a $d$-dimensional Poisson spatial point process with constant density $\lambda>0$. Fix a realization of $\{x_{i}\}_{i\in\mathbb{N}}$ and suppose that the target is obtained by making each point $x_{i}\in\R^{d}$ into a ball of radius $l>0$,
\begin{align*}
U
:=\{x\in\R^{d}:\|x-x_{i}\|\le l\quad\text{for some }i\in\mathbb{N}\}.
\end{align*}
Prior work often considers the case of sparse targets, which means that $\lambda l^{d}V_{d}\ll1$, where $V_{d}:=\pi^{d/2}/\Gamma(1+d/2)>0$ is the $d$-dimensional volume of a unit sphere.

If the support of the initial distribution of ${{X}}(0)$ does not intersect the target (see \eqref{away}), then Theorem~\ref{key} applies. To approximate the rate $\rho$ in \eqref{tae}, we use that $\P({{B}}(s)\in U)$ vanishes exponentially as $s\to0+$ and $\P({{B}}(s)\in U)\to \lambda l^{d}V_{d}\in(0,1)$ as $s\to\infty$, since $\lambda l^{d}V_{d}$ is the fraction of space occupied by targets. The characteristic  distance between neighboring $x_{i}$ and $x_{j}$ is $L:=(\lambda V_{d})^{-1/d}\gg l$ and so the characteristic timescale when ${{B}}$ reaches the target is $L^{2}$ (${{B}}$ has unit diffusivity). Hence, if we approximate $\P({{B}}(s)\in U)$ by 0 for $s<L^{2}$ and by $\lambda l^{d}V_{d}$ for $s>L^{2}$, then we obtain
\begin{align}\label{rhoapprox}
\rho
&\approx K\int_{L^{2}}^{\infty}\lambda l^{d}V_{d}\frac{\alpha/2 }{\Gamma(1-\frac{\alpha}{2})}\frac{1}{s^{1+\alpha/2}}\,\dd s
=\frac{V_{d}^{1+\alpha/d}}{\Gamma(1-\frac{\alpha}{2})}{{K}}l^{d}\lambda^{1+\alpha/d}.
\end{align}
If we define $X$ via $X(t):=B(S(t))$ where $S$ is the tempered stable subordinator in \eqref{tempered} with $\mu>0$, then the analysis above holds and the approximation in \eqref{rhoapprox} is 
\begin{align*}
\rho
&\approx K\int_{L^{2}}^{\infty}\lambda l^{d}V_{d}\frac{\alpha/2 }{\Gamma(1-\frac{\alpha}{2})}\frac{e^{-\mu s}}{s^{1+\alpha/2}}\,\dd s
=\frac{Kl^{d}V_{d}\alpha\lambda\mu^{\alpha/2}\Gamma(-\frac{\alpha}{2},(V_{d}\lambda)^{-2/d}\mu)}{2\Gamma(1-\frac{\alpha}{2})}.
\end{align*}

\subsection{Stochastic simulation algorithm}\label{simulation}

We now give the stochastic simulation algorithm used to generate FHTs of L{\'e}vy flights. Given a discrete time step $\Delta t>0$, we generate a statistically exact path of the $(\alpha/2)$-stable subordinator $S=\{S(t)\}_{t\ge0}$ on the discrete time grid $\{t_{k}\}_{k\in\mathbb{N}}$ with $t_{k}=k\Delta t$ via
\begin{align*}
S(t_{k+1})
=S(t_{k})+(\Delta t)^{2/\alpha}\Theta_{k},\quad k\ge0,
\end{align*}
where $S(t_{0})=S(0)=0$ and $\{\Theta_{k}\}_{k\in\mathbb{N}}$ is an iid sequence of realizations of \cite{carnaffan2017}
\begin{align*}
\Theta
=\frac{\sin(\gamma(V+\pi/2)}{(\cos(V))^{1/\gamma}}\bigg(\frac{\cos(V-\gamma(V+\pi/2))}{E}\bigg)^{\frac{1-\gamma}{\gamma}},\quad\text{with }\gamma:=\alpha/2\in(0,1),
\end{align*}
where $V$ is uniformly distributed on $(-\pi/2,\pi/2)$ and $E$ is an independent exponential random variable with $\E[E]=1$. This allows us to generate a statistically exact path of the Brownian motion $\{{{B}}(s)\}_{s\ge0}$ on the (random) discrete time grid $\{S(t_{k})\}_{k\in\mathbb{N}}$ via
\begin{align*}
{{B}}(S(t_{k+1}))
={{B}}(S(t_{k}))+\sqrt{2({{K}}\Delta t)^{2/\alpha}\Theta_{k}}\xi_{k},\quad k\ge0,
\end{align*}
where $\{\xi_{k}\}_{k\in\mathbb{Z}}$ is an iid sequence of standard $d$-dimensional Gaussian vectors. Finally, we obtain a statistically exact path of the L{\'e}vy process ${{X}}=\{{{X}}(t)\}_{t\ge0}$ on the discrete time grid $\{t_{k}\}_{k\in\mathbb{N}}$ via ${{X}}(t_{k})
={{B}}(S(t_{k}))$ for $k\ge0$. The FHT $\tau$ to $U\subset\R^{d}$ is then approximated by $\overline{k}
:=\min\{k\Delta t\ge0:{{X}}(t_{k})\in U\}$.

Paths of the gamma subordinated Brownian motion in  section~\ref{othersubs} are simulated using the same method, except that $\{\Theta_{k}\}_{k\in\mathbb{N}}$ is an iid sequence of realizations of gamma random variables with shape $C\Delta t>0$ and rate $\mu>0$. The data in Figures~\ref{fighalf}-\ref{figgamma} is computed from $10^{5}$ independent trials with $\Delta t=10^{-5}$.
 
\section{Discussion}

Most studies of search processes measure the speed of search in terms of the FHT of a single searcher. In this paper, we considered the scenario in which there are $N\gg1$ iid searchers and studied the FHT of the fastest searcher to find the target. Our analysis involved finding the short-time distribution of the FHT of a single searcher and using this to find the distribution and moments of the FHT for the fastest searcher. Our results apply to searchers whose paths follow a subordinate Brownian motion, which is any process obtained by composing a Brownian motion with a L{\'e}vy subordinator. We were primarily interested in the case that the searchers move by L{\'e}vy flights, which is a prototypical model for superdiffusive search \cite{dubkov2008}.

Previous analysis of extreme FHTs has focused on diffusion, which began with the work of Weiss, Shuler, and Lindenberg in 1983 \cite{weiss1983}. The $1/N$ decay of mean extreme FHTs for subordinate Brownian motion contrasts sharply with the well-known $1/\ln N$ decay of extreme FHTs for diffusion (compare \eqref{decay} and \eqref{diff0}). See the Introduction section for more on how extreme statistics and large deviation theory for subordinate Brownian motion compare to diffusion. Our results also contrast with results on extreme FHTs of subdiffusive processes modeled by a time fractional Fokker-Planck equation \cite{lawley2020sub}. For searchers exploring a discrete space, an interesting recent study analyzed extreme FHTs for L{\'e}vy  walks on the two-dimensional integer lattice \cite{clementi2020}, which was motivated by the L{\'e}vy flight foraging hypothesis described in the Introduction section above. Other works investigating extreme FHTs on discrete state networks include \cite{weng2017, feinerman2012} in discrete time and \cite{lawley2020networks} in continuous-time. 

Biological search processes are often modeled by superdiffusive L{\'e}vy walks \cite{reynolds2018}, which are similar to L{\'e}vy  flights but move with a finite velocity \cite{shlesinger1986}. In particular, L{\'e}vy  walks follow ballistic flights of uniformly distributed random directions and constant speed, and the lengths of the flights are chosen from a probability density with the slow power law decay in \eqref{pl}. L{\'e}vy  walks are thus similar to run-and-tumble processes, except run-and-tumble models typically assume the distance of each ballistic flight (i.e.\ a ``run'') is chosen from an exponential distribution. The choice of an exponential distribution makes a run-and-tumble a piecewise deterministic Markov process. While L{\'e}vy  walks are not Markovian, they are nonetheless piecewise deterministic in the sense that the motion is deterministic (constant velocity in a fixed direction) between turns. Extreme FHTs of piecewise deterministic processes were analyzed in \cite{lawley2021pdmp}, and it would be interesting to apply that theory to L{\'e}vy walks. 

\section{Appendix}

In this appendix, we prove the results in the main text.

\begin{lemma}\label{cpp}
Assume $S=\{S(t)\}_{t\ge0}$ is a compound Poisson process plus a drift, meaning its Laplace exponent is in \eqref{le} with $b\ge0$ and $\int_{0}^{\infty}\,\nu(\dd z)\in(0,\infty)$. If $F:[0,\infty)\to[0,1]$ is continuous and satisfies \eqref{fv}, then \eqref{fc} holds.
\end{lemma}

\begin{proof}[Proof of Lemma~\ref{cpp}]
By assumption, we have that $S(t)=bt+\sum_{m=1}^{M(t)}Z_{m}$, where $M=\{M(t)\}_{t\ge0}$ is a Poisson process with rate $\lambda=\int_{0}^{\infty}\,\nu(\dd z)\in(0,\infty)$ and $\{Z_{m}\}_{m\ge1}$ are iid nonnegative random variables independent of $M$. In this case, the probability measure of $Z_{m}$ is $\nu(\dd z)/\lambda$. Decomposing the mean based on the value of $M(t)$ yields
\begin{align*}
\E[F(S(t))]
=\E[F(S(t))1_{M(t)=0}]+\E[F(S(t))1_{M(t)=1}]+\E[F(S(t))1_{M(t)\ge2}],
\end{align*}
where $1_{A}$ denotes the indicator function on an event $A$. Since $M(t)$ is a Poisson random variable with mean $\lambda t$ and $F$ is bounded, we have that $\E[F(S(t))1_{M(t)\ge2}]=o(t)$ as $t\to0+$. Furthermore, since $M$ and $Z_{1}$ are independent, we have that
\begin{align*}
\E[F(S(t))1_{M(t)=1}]
&=\P(M(t)=1)\E[F(bt+Z_{1})]
=\lambda t e^{-\lambda t}\E[F(bt+Z_{1})],\\
\E[F(S(t))1_{M(t)=0}]
&=\P(M(t)=0)\E[F(bt)]
=e^{-\lambda t}F(bt).
\end{align*}
Since $F$ is bounded, $F$ is continuous, and $\int_{0}^{\infty}\,\nu(\dd s)<\infty$, we complete the proof by applying the Lebesgue dominated convergence to conclude
\begin{align*}
\E[F(bt+Z_{1})]
=\frac{1}{\lambda}\int_{0}^{\infty}F(bt+s)\,\nu(\dd s)
\to\frac{1}{\lambda}\int_{0}^{\infty}F(s)\,\nu(\dd s)\quad\text{as }t\to0+.
\end{align*}
\end{proof}

\begin{proof}[Proof of Proposition~\ref{ias2}]

The boundedness of $F$ and \eqref{fv} ensure that the integral in \eqref{fc} is finite. Let $\eps=2^{-j}>0$ for some $j\in\{0,1,2,\dots\}$ and define 
\begin{align}\label{se}
\begin{split}
S_{[\eps,\infty)}(t)
&:=bt+\iint_{z\in[\eps,\infty),\,t'\in[0,t]} z\,\mathbf{N}(\dd t',\dd z),\\
S_{(0,\eps)}(t)
&:=\iint_{z\in(0,\eps),\,t'\in[0,t]} z\,\mathbf{N}(\dd t',\dd z),
\end{split}
\end{align}
where $\mathbf{N}$ is a Poisson point process on the first quadrant with intensity measure $\dd t'\,\nu(\dd z)$. The process $S$ can then be written as $S(t)=S_{[\eps,\infty)}(t)+S_{(0,\eps)}(t)$. Since $F$ is Lipschitz, there exists a constant $\kappa>0$ so that
\begin{align}\label{2side}
\begin{split}
&\E[F(S_{[\eps,\infty)}(t))]
-\kappa\E[S_{(0,\eps)}(t))]
\le\E[F(S(t))]\\
&\qquad\qquad\qquad\le \E[F(S_{[\eps,\infty)}(t))]
+\kappa\E[S_{(0,\eps)}(t))]\quad\text{for all }t>0.
\end{split}
\end{align}
Since $S_{[\eps,\infty)}$ is a compound Poisson process plus a drift, Lemma~\ref{cpp} implies that
\begin{align}\label{first}
\lim_{t\to0+}t^{-1}\E[F(S_{[\eps,\infty)}(t))]
=\rho_{\eps}
:=bF'(0)+\int_{\eps}^{\infty}F(s)\,\nu(\dd s)
<\infty.
\end{align}

To handle the terms in \eqref{2side} involving $S_{(0,\eps)}(t)$, recall that $\eps=2^{-j}$ and observe that a dyadic partitioning of the interval $(0,\eps)$ yields
\begin{align*}
S_{(0,\eps)}(t)
&:=\iint_{z\in(0,\eps),\,t'\in[0,t]} z\,\mathbf{N}(\dd t',\dd z)
\le\sum_{k=j}^{\infty}2^{-k}\mathbf{N}([0,t]\times[2^{-k-1},2^{-k}]).
\end{align*}
Since $\mathbf{N}$ is a Poisson point process, we have that
\begin{align*}
2^{-k}\E\big[\mathbf{N}([0,t]\times[2^{-k-1},2^{-k}])\big]
=2t\int_{2^{-k-1}}^{2^{-k}}2^{-k-1}\,\nu(\dd z)
\le2t\int_{2^{-k-1}}^{2^{-k}}z\,\nu(\dd z).
\end{align*}
Therefore,
\begin{align}\label{nicebound}
\E[S_{(0,\eps)}(t)]
\le2t\int_{0}^{\eps}z\,\nu(\dd z).
\end{align}
Combining \eqref{2side} with \eqref{first} and \eqref{nicebound} yields
\begin{align*}
&\rho_{\eps}
-2\kappa\int_{0}^{\eps}z\,\nu(\dd z)
\le\liminf_{t\to0+}\frac{\E[F(S(t))]}{t}
\le\limsup_{t\to0+}\frac{\E[F(S(t))]}{t}
\le\rho_{\eps}
+2\kappa\int_{0}^{\eps}z\,\nu(\dd z).
\end{align*}
Since these bounds converge to $\rho$ as $\eps\to0+$, the proof is complete. 
\end{proof}

\begin{lemma}\label{mono}
Let $H:[0,\infty)\to[0,1]$ be nondecreasing and satisfy \eqref{fv}. Then $\limsup_{t\to0+}\E[H(S(t))]/t<\infty$.
\end{lemma}

\begin{proof}[Proof of Lemma~\ref{mono}]
Using the definitions in \eqref{se}, we have that
\begin{align*}
H(S(t))
=H(S_{[\eps,\infty)}(t)+S_{(0,\eps)}(t))
&\le H(2 S_{[\eps,\infty)}(t)) +H(2S_{(0,\eps)}(t)).
\end{align*}
Since $2 S_{[\eps,\infty)}(t)$ is a compound Poisson process plus a drift, Lemma~\ref{cpp} ensures that $
\lim_{t\to0+}\E[H(2 S_{[\eps,\infty)}(t))]/t
<\infty$. Since $H$ satisfies \eqref{fv}, there exists an $s_{0}\in(0,1]$ and a $\theta\ge1$ so that $H(s)\le \theta s$ for all $s\in(0,s_{0}]$. Therefore, $H(s)\le \theta s/s_{0}$ for all $s\ge0$. The proof is complete since \eqref{nicebound} implies
\begin{align*}
\E[H(2S_{(0,\eps)}(t))]
\le\frac{2\theta}{s_{0}}\E[S_{(0,\eps)}(t)]
\le\frac{4\theta t}{s_{0}}\int_{0}^{\eps}z\,\nu(\dd z).
\end{align*}
\end{proof}

\begin{proof}[Proof of Theorem~\ref{key}]

Define $F(s):=\P(B(s)+X(0)\in U)\in[0,1]$ for $s\ge0$. Using the independence of $B$ and $X(0)$, we have
\begin{align}\label{fa}
F(s)
=\frac{1}{(4\pi s)^{d/2}}\iint_{U\times U_{0}}\exp\Big(\frac{-\|x-x_{0}\|^{2}}{4s}\Big)\,\mu_{0}(\dd x_{0})\,\dd x,\quad\text{if }s>0,
\end{align}
where $\mu_{0}$ is the probability measure of $X(0)$ with support $U_{0}\subset\R^{d}$. Using standard results for interchanging differentiation with integration (for example, see Theorem A.5.3 in \cite{durrett2019}), $F(s)$ is infinitely differentiable and each derivative is bounded. Furthermore, \eqref{away} ensures that $F(0)=F'(0)=0$ and thus Proposition~\ref{ias2} implies
\begin{align}\label{ef}
\lim_{t\to0+}\frac{\P({{X}}(t)\in U)}{t}
=\lim_{t\to0+}\frac{\E[F(s)]}{t}
=\rho
:=\int_{0}^{\infty}\P({{B}}(s)+X(0)\in U)\,\nu(\dd s).
\end{align}
In the first equality in \eqref{ef}, we have used the independence of $B$, $S$, and $X(0)$. Note that $\rho\in(0,\infty)$. Indeed, Proposition~\ref{ias2} implies $\rho<\infty$. Further, $\rho>0$ by (i) the assumption in \eqref{nontrivial}, (ii) the fact that ${{B}}(s)\in\R^{d}$ is a Gaussian random variable with variance proportional to $s>0$, and (iii) $U$ has strictly positive Lebesgue measure (since $U$ is nonempty and the closure of its interior).

To complete the proof, we therefore need to show that
\begin{align}\label{wts}
\lim_{t\to0+}t^{-1}\P(\tau\le t)
=\lim_{t\to0+}t^{-1}\P({{X}}(t)\in U).
\end{align}
For $t>0$, define the enlarged target $U^{\delta(t)}
:=\{x\in\R^{d}:\inf_{y\in U}\|x-y\|\le\delta(t)\}$, where we set $\delta(t):=t^{1/4}>0$ in order to satisfy
\begin{align}\label{delta}
\lim_{t\to0+}\delta(t)=0
\quad\text{and}\quad
\lim_{t\to0+}\delta(t)t^{-1/2}=\infty.
\end{align}
Decomposing the event $\tau\le t$ based on the position of ${{X}}(t)$ yields
\begin{align*}
\P(\tau\le t)
&=\P({{X}}(t)\in U)
+\P(\tau\le t,{{X}}(t)\in U^{\delta(t)}\backslash U)
+\P(\tau\le t,{{X}}(t)\notin U^{\delta(t)}).
\end{align*}
Therefore, showing \eqref{wts} amounts to showing that
\begin{align}
\lim_{t\to0+}t^{-1}\P(\tau\le t,{{X}}(t)\in U^{\delta(t)}\backslash U)
&=0=\lim_{t\to0+}t^{-1}\P(\tau\le t,{{X}}(t)\notin U^{\delta(t)}).\label{staycloseandmoveaway}
\end{align}

We first prove the first equality in \eqref{staycloseandmoveaway}. Since ${{X}}(t)={{B}}(S(t))+X(0)$ and ${{B}}$, $X(0)$, and $S$ are independent, integrating over the possible values of $S(t)$ yields
\begin{align*}
\P(\tau\le t,{{X}}(t)\in U^{\delta(t)}\backslash U)
&\le\P({{X}}(t)\in U^{\delta(t)}\backslash U)
=\E[F_{0}(S(t);t)],
\end{align*}
where $F_{0}(s;t):=\P({{B}}(s)+X(0)\in U^{\delta(t)}\backslash U)$. By the assumption in \eqref{away}, we may take $t_{0}$ sufficiently small so that $U^{\delta(t_{0})}\cap U_{0}=\varnothing$. Therefore, if $t\in(0,t_{0}]$, then $F_{0}(s;t)$ satisfies the assumptions of Proposition~\ref{ias2} (by the same argument used for $F(s)$ in \eqref{fa}). Therefore, Proposition~\ref{ias2} implies that we may take $t$ sufficiently small so that,
\begin{align*}
t^{-1}\P(\tau\le t,{{X}}(t)\in U^{\delta(t)}\backslash U)
\le2\int_{0}^{\infty}\P({{B}}(s)+X(0)\in U^{\delta(t_{0})}\backslash U)\,\nu(\dd s)<\infty.
\end{align*}
Now, it is immediate that $\P({{B}}(s)+X(0)\in U^{\delta(t_{0})}\backslash U)\to0$ as $t_{0}\to0$ for each $s\ge0$. Hence, the Lebesgue dominated convergence theorem implies
\begin{align*}
\lim_{t_{0}\to0+}\int_{0}^{\infty}\P({{B}}(s)+X(0)\in U^{\delta(t_{0})}\backslash U)\,\nu(\dd s)=0,
\end{align*}
and thus the first equality in \eqref{staycloseandmoveaway} holds. Turning to the second equality in  \eqref{staycloseandmoveaway}, conditioning that $\tau\le t$ implies
\begin{align*}
\P(\tau\le t,{{X}}(t)\notin U^{\delta(t)})
=\P({{X}}(t)\notin U^{\delta(t)}\,|\,\tau\le t)\P(\tau\le t),
\end{align*}
and the fact that $\wt\le\tau$ almost surely and Lemma~\ref{mono} imply
\begin{align*}
\limsup_{t\to0+}t^{-1}\P(\tau\le t)
\le \limsup_{t\to0+}t^{-1}\P(\wt\le t)
<\infty,
\end{align*}
since $\P(\wt\le t)=\E[H(S(t)]$ where $H(s):=\P(\sigma\le s)$ is nondecreasing. Next, it follows from the strong Markov property \cite{bertoin1996} that
\begin{align*}
\P({{X}}(t)\notin U^{\delta(t)}\,|\,\tau\le t)
\le\sup_{r\in(0,t]}\P_{0}(\|{{X}}(r)\|\ge\delta(t)),
\end{align*}
where $\P_{0}$ denotes the probability measure conditioned that ${{X}}(0)=0$. Again using that $B$ and $S$ are independent, we have that 
\begin{align}\label{ab3}
\begin{split}
\P_{0}(\|{{X}}(r)\|\ge\delta(t))
=\E[F_{1}(S(r);t)]
\le\E[F_{1}(S(t);t)],\quad\text{if }r\in[0,t],
\end{split}
\end{align}
since $F_{1}(s;t):=\P(\|B(s)\|\ge\delta(t))$ is an increasing function of $s$ and $S$ is almost surely nondecreasing.
Define
\begin{align}\label{delta1}
\delta_{1}(t)
:=(1+b)t>0,
\end{align}
and observe that \eqref{ab3} implies that for $r\in(0,t]$,
\begin{align*}
\begin{split}
\P_{0}(\|{{X}}(r)\|\ge\delta(t))
&\le\E[F_{1}(S(t);t)1_{S(t)<\delta_{1}(t)}]
+\E[F_{1}(S(t);t)1_{S(t)\ge\delta_{1}(t)}]. 
\end{split}
\end{align*}
Since $S(t)/t$ converges in probability to $b\ge0$ as $t\to0+$ \cite{bertoin1996}, we have that 
\begin{align*}
\E[F_{1}(S(t);t)1_{S(t)\ge\delta_{1}(t)}]
\le\P(S(t)\ge\delta_{1}(t))
=\P(S(t)\ge(1+b)t)\to0\quad\text{as }t\to0+.
\end{align*}
Next, since $F_{1}(s;t)$ is an increasing function of $s$, we have that
\begin{align*}
\E[F_{1}(S(t);t)1_{S(t)<\delta_{1}(t)}]
\le F_{1}(\delta_{1}(t);t)
=\P(\|B(\delta_{1}(t))\|\ge\delta(t)).
\end{align*}
The Brownian scaling in \eqref{brownianscaling} and the choices of $\delta(t)$ in \eqref{delta} and $\delta_{1}(t)$ in \eqref{delta1} imply
\begin{align*}
\P(\|B(\delta_{1}(t))\|\ge\delta(t))
=\P(\|B(1)\|\ge\delta(t)(\delta_{1}(t))^{-1/2})
\to0\quad\text{as }t\to0+.
\end{align*}
Hence, the second equality in \eqref{staycloseandmoveaway} holds and the proof is complete.
\end{proof}


\bibliography{library.bib}

\begin{thebibliography}{10}

\bibitem{benichou2011rev}
O~B{\'e}nichou, C~Loverdo, M~Moreau, and R~Voituriez.
\newblock Intermittent search strategies.
\newblock {\em Rev Mod Phys}, 83(1):81, 2011.

\bibitem{frost2001}
JR~Frost and Lawrence~D Stone.
\newblock Review of search theory: advances and applications to search and
  rescue decision support.
\newblock {\em US Department of Transportation}, 2001.

\bibitem{morse1956}
Phillip~M Morse and G~Kendall.
\newblock How to hunt a submarine: The world of mathematics, 1956.

\bibitem{reynolds2018}
Andy~M Reynolds.
\newblock Current status and future directions of {L}{\'e}vy walk research.
\newblock {\em Biology open}, 7(1), 2018.

\bibitem{viswanathan2008}
GM~Viswanathan, EP~Raposo, and MGE Da~Luz.
\newblock {L}{\'e}vy flights and superdiffusion in the context of biological
  encounters and random searches.
\newblock {\em Physics of Life Reviews}, 5(3):133--150, 2008.

\bibitem{lomholt2005}
Michael~A Lomholt, Tobias Ambj{\"o}rnsson, and Ralf Metzler.
\newblock Optimal target search on a fast-folding polymer chain with volume
  exchange.
\newblock {\em Physical review letters}, 95(26):260603, 2005.

\bibitem{kao1996}
Ming-Yang Kao, John~H Reif, and Stephen~R Tate.
\newblock Searching in an unknown environment: An optimal randomized algorithm
  for the cow-path problem.
\newblock {\em Information and Computation}, 131(1):63--79, 1996.

\bibitem{shlesinger2006}
Michael~F Shlesinger.
\newblock Search research.
\newblock {\em Nature}, 443(7109):281--282, 2006.

\bibitem{metzler2004}
Ralf Metzler and Joseph Klafter.
\newblock The restaurant at the end of the random walk: recent developments in
  the description of anomalous transport by fractional dynamics.
\newblock {\em Journal of Physics A: Mathematical and General}, 37(31):R161,
  2004.

\bibitem{dubkov2008}
Alexander~A Dubkov, Bernardo Spagnolo, and Vladimir~V Uchaikin.
\newblock {L}{\'e}vy flight superdiffusion: an introduction.
\newblock {\em International Journal of Bifurcation and Chaos},
  18(09):2649--2672, 2008.

\bibitem{meerschaert2019}
Mark~M Meerschaert and Alla Sikorskii.
\newblock {\em Stochastic models for fractional calculus}, volume~43.
\newblock Walter de Gruyter GmbH \& Co KG, 2019.

\bibitem{lischke2020}
Anna Lischke, Guofei Pang, Mamikon Gulian, Fangying Song, Christian Glusa,
  Xiaoning Zheng, Zhiping Mao, Wei Cai, Mark~M Meerschaert, Mark Ainsworth,
  et~al.
\newblock What is the fractional {L}aplacian? {A} comparative review with new
  results.
\newblock {\em Journal of Computational Physics}, 404:109009, 2020.

\bibitem{zaburdaev2015}
V~Zaburdaev, S~Denisov, and J~Klafter.
\newblock {L}{\'e}vy walks.
\newblock {\em Reviews of Modern Physics}, 87(2):483, 2015.

\bibitem{shlesinger1986}
Michael~F Shlesinger and Joseph Klafter.
\newblock {L}{\'e}vy walks versus {L}{\'e}vy flights.
\newblock In {\em On growth and form}, pages 279--283. Springer, 1986.

\bibitem{viswanathan1996}
Gandhimohan~M Viswanathan, V~Afanasyev, SV~Buldyrev, EJ~Murphy, PA~Prince, and
  H~Eugene Stanley.
\newblock {L}{\'e}vy flight search patterns of wandering albatrosses.
\newblock {\em Nature}, 381(6581):413--415, 1996.

\bibitem{viswanathan1999}
Gandimohan~M Viswanathan, Sergey~V Buldyrev, Shlomo Havlin, MGE Da~Luz,
  EP~Raposo, and H~Eugene Stanley.
\newblock Optimizing the success of random searches.
\newblock {\em nature}, 401(6756):911--914, 1999.

\bibitem{ramos2004}
Gabriel Ramos-Fern{\'a}ndez, Jos{\'e}~L Mateos, Octavio Miramontes, Germinal
  Cocho, Hern{\'a}n Larralde, and Barbara Ayala-Orozco.
\newblock {L}{\'e}vy walk patterns in the foraging movements of spider monkeys
  (ateles geoffroyi).
\newblock {\em Behavioral ecology and Sociobiology}, 55(3):223--230, 2004.

\bibitem{boyer2006}
Denis Boyer, Gabriel Ramos-Fern{\'a}ndez, Octavio Miramontes, Jos{\'e}~L
  Mateos, Germinal Cocho, Hern{\'a}n Larralde, Humberto Ramos, and Fernando
  Rojas.
\newblock Scale-free foraging by primates emerges from their interaction with a
  complex environment.
\newblock {\em Proceedings of the Royal Society B: Biological Sciences},
  273(1595):1743--1750, 2006.

\bibitem{atkinson2002}
RPD Atkinson, CJ~Rhodes, DW~Macdonald, and RM~Anderson.
\newblock Scale-free dynamics in the movement patterns of jackals.
\newblock {\em Oikos}, 98(1):134--140, 2002.

\bibitem{sims2006}
David~W Sims, Matthew~J Witt, Anthony~J Richardson, Emily~J Southall, and
  Julian~D Metcalfe.
\newblock Encounter success of free-ranging marine predator movements across a
  dynamic prey landscape.
\newblock {\em Proceedings of the Royal Society B: Biological Sciences},
  273(1591):1195--1201, 2006.

\bibitem{bartumeus2003}
Frederic Bartumeus, Francesc Peters, Salvador Pueyo, Celia Marras{\'e}, and
  Jordi Catalan.
\newblock Helical {L}{\'e}vy walks: adjusting searching statistics to resource
  availability in microzooplankton.
\newblock {\em Proceedings of the National Academy of Sciences},
  100(22):12771--12775, 2003.

\bibitem{reverey2015}
Julia~F Reverey, Jae-Hyung Jeon, Han Bao, Matthias Leippe, Ralf Metzler, and
  Christine Selhuber-Unkel.
\newblock Superdiffusion dominates intracellular particle motion in the
  supercrowded cytoplasm of pathogenic acanthamoeba castellanii.
\newblock {\em Scientific reports}, 5(1):1--14, 2015.

\bibitem{pavlyukevich2007}
Ilya Pavlyukevich.
\newblock {L}{\'e}vy flights, non-local search and simulated annealing.
\newblock {\em Journal of Computational Physics}, 226(2):1830--1844, 2007.

\bibitem{edwards2007}
Andrew~M Edwards, Richard~A Phillips, Nicholas~W Watkins, Mervyn~P Freeman,
  Eugene~J Murphy, Vsevolod Afanasyev, Sergey~V Buldyrev, Marcos~GE da~Luz,
  Ernesto~P Raposo, H~Eugene Stanley, et~al.
\newblock Revisiting {L}{\'e}vy flight search patterns of wandering
  albatrosses, bumblebees and deer.
\newblock {\em Nature}, 449(7165):1044--1048, 2007.

\bibitem{levernier2020}
Nicolas Levernier, Johannes Textor, Olivier B{\'e}nichou, and Rapha{\"e}l
  Voituriez.
\newblock Inverse square {L}{\'e}vy walks are not optimal search strategies for
  $d\ge 2$.
\newblock {\em Physical review letters}, 124(8):080601, 2020.

\bibitem{buldyrev2021}
SV~Buldyrev, EP~Raposo, Frederic Bartumeus, S~Havlin, FR~Rusch, MGE da~Luz, and
  GM~Viswanathan.
\newblock {C}omment on ``{I}nverse square {L}{\'e}vy walks are not optimal
  search strategies for $d\ge2$''.
\newblock {\em Physical Review Letters}, 126(4):048901, 2021.

\bibitem{levernier2021reply}
Nicolas Levernier, Johannes Textor, Olivier B{\'e}nichou, and Rapha{\"e}l
  Voituriez.
\newblock {R}eply to ``{C}omment on `{I}nverse square {L}{\'e}vy walks are not
  optimal search strategies for $d\ge2$'''.
\newblock {\em Physical Review Letters}, 126(4):048902, 2021.

\bibitem{eliazar2004}
Iddo Eliazar and Joseph Klafter.
\newblock On the first passage of one-sided {L}{\'e}vy motions.
\newblock {\em Physica A: Statistical Mechanics and its Applications},
  336(3-4):219--244, 2004.

\bibitem{koren2007}
Tal Koren, Michael~A Lomholt, Aleksei~V Chechkin, Joseph Klafter, and Ralf
  Metzler.
\newblock Leapover lengths and first passage time statistics for {L}{\'e}vy
  flights.
\newblock {\em Physical review letters}, 99(16):160602, 2007.

\bibitem{koren2007first}
T~Koren, AV~Chechkin, and J~Klafter.
\newblock On the first passage time and leapover properties of {L}{\'e}vy
  motions.
\newblock {\em Physica A: Statistical Mechanics and its Applications},
  379(1):10--22, 2007.

\bibitem{gao2014}
Ting Gao, Jinqiao Duan, Xiaofan Li, and Renming Song.
\newblock Mean exit time and escape probability for dynamical systems driven by
  {L}{\'e}vy noises.
\newblock {\em SIAM Journal on Scientific Computing}, 36(3):A887--A906, 2014.

\bibitem{palyulin2019}
Vladimir~V Palyulin, George Blackburn, Michael~A Lomholt, Nicholas~W Watkins,
  Ralf Metzler, Rainer Klages, and Aleksei~V Chechkin.
\newblock First passage and first hitting times of {L}{\'e}vy flights and
  {L}{\'e}vy walks.
\newblock {\em New Journal of Physics}, 21(10):103028, 2019.

\bibitem{wardak2020}
Asem Wardak.
\newblock First passage leapovers of {L}{\'e}vy flights and the proper
  formulation of absorbing boundary conditions.
\newblock {\em Journal of Physics A: Mathematical and Theoretical},
  53(37):375001, 2020.

\bibitem{schuss2019}
Z.~Schuss, K.~Basnayake, and D.~Holcman.
\newblock Redundancy principle and the role of extreme statistics in molecular
  and cellular biology.
\newblock {\em Physics of Life Reviews}, January 2019.

\bibitem{schoener1971}
Thomas~W Schoener.
\newblock Theory of feeding strategies.
\newblock {\em Annual review of ecology and systematics}, 2(1):369--404, 1971.

\bibitem{traniello1977}
James~FA Traniello.
\newblock Recruitment behavior, orientation, and the organization of foraging
  in the carpenter ant camponotus pennsylvanicus degeer (hymenoptera:
  Formicidae).
\newblock {\em Behavioral Ecology and Sociobiology}, 2(1):61--79, 1977.

\bibitem{holldobler1990}
Bert H{\"o}lldobler, Edward~O Wilson, et~al.
\newblock {\em The ants}.
\newblock Harvard University Press, 1990.

\bibitem{wenzel1991}
John~W Wenzel and John Pickering.
\newblock Cooperative foraging, productivity, and the central limit theorem.
\newblock {\em Proceedings of the National Academy of Sciences}, 88(1):36--38,
  1991.

\bibitem{jarvis1998}
Jennifer~UM Jarvis, Nigel~C Bennett, and Andrew~C Spinks.
\newblock Food availability and foraging by wild colonies of damaraland
  mole-rats (cryptomys damarensis): implications for sociality.
\newblock {\em Oecologia}, 113(2):290--298, 1998.

\bibitem{torney2009}
Colin Torney, Zoltan Neufeld, and Iain~D Couzin.
\newblock Context-dependent interaction leads to emergent search behavior in
  social aggregates.
\newblock {\em Proceedings of the National Academy of Sciences},
  106(52):22055--22060, 2009.

\bibitem{torney2011}
Colin~J Torney, Andrew Berdahl, and Iain~D Couzin.
\newblock Signalling and the evolution of cooperative foraging in dynamic
  environments.
\newblock {\em PLoS Comput Biol}, 7(9):e1002194, 2011.

\bibitem{feinerman2012}
Ofer Feinerman, Amos Korman, Zvi Lotker, and Jean-S{\'e}bastien Sereni.
\newblock Collaborative search on the plane without communication.
\newblock In {\em Proceedings of the 2012 ACM symposium on Principles of
  distributed computing}, pages 77--86, 2012.

\bibitem{lawley2020esp1}
S~D Lawley and J~B Madrid.
\newblock A probabilistic approach to extreme statistics of brownian escape
  times in dimensions 1, 2, and 3.
\newblock {\em J Nonlinear Sci}, pages 1--21, 2020.

\bibitem{ro2017}
S~Ro and Y~W Kim.
\newblock Parallel random target searches in a confined space.
\newblock {\em Phys Rev E}, 96(1):012143, 2017.

\bibitem{clementi2020}
Andrea Clementi, Francesco d'Amore, George Giakkoupis, and Emanuele Natale.
\newblock On the search efficiency of parallel {L}{\'e}vy walks on
  $\mathbb{Z}^{2}$.
\newblock {\em arXiv preprint arXiv:2004.01562}, 2020.

\bibitem{weiss1983}
G~H Weiss, K~E Shuler, and K~Lindenberg.
\newblock Order statistics for first passage times in diffusion processes.
\newblock {\em J Stat Phys}, 31(2):255--278, 1983.

\bibitem{lawley2020uni}
S~D Lawley.
\newblock Universal formula for extreme first passage statistics of diffusion.
\newblock {\em Phys Rev E}, 101(1):012413, 2020.

\bibitem{varadhan1967}
Sathamangalam R~Srinivasa Varadhan.
\newblock Diffusion processes in a small time interval.
\newblock {\em Commun Pure Appl Math}, 20(4):659--685, 1967.

\bibitem{bertoin1996}
Jean Bertoin.
\newblock {\em {L}{\'e}vy processes}, volume 121.
\newblock Cambridge {U}niversity {P}ress, 1996.

\bibitem{kim2012}
Panki Kim, Renming Song, and Zoran Vondra{\v{c}}ek.
\newblock Two-sided green function estimates for killed subordinate brownian
  motions.
\newblock {\em Proceedings of the London Mathematical Society},
  104(5):927--958, 2012.

\bibitem{madrid2020comp}
Jacob~B Madrid and Sean~D Lawley.
\newblock Competition between slow and fast regimes for extreme first passage
  times of diffusion.
\newblock {\em Journal of Physics A: Mathematical and Theoretical},
  53(33):335002, 2020.

\bibitem{lawley2020dist}
S~D Lawley.
\newblock Distribution of extreme first passage times of diffusion.
\newblock {\em Journal of Mathematical Biology}, 2020.

\bibitem{getoor1961}
RK~Getoor.
\newblock First passage times for symmetric stable processes in space.
\newblock {\em Transactions of the American Mathematical Society},
  101(1):75--90, 1961.

\bibitem{carnaffan2017}
Sean Carnaffan and Reiichiro Kawai.
\newblock Solving multidimensional fractional {F}okker--{P}lanck equations via
  unbiased density formulas for anomalous diffusion processes.
\newblock {\em SIAM Journal on Scientific Computing}, 39(5):B886--B915, 2017.

\bibitem{lawley2020sub}
Sean~D Lawley.
\newblock Extreme statistics of anomalous subdiffusion following a fractional
  fokker--planck equation: subdiffusion is faster than normal diffusion.
\newblock {\em Journal of Physics A: Mathematical and Theoretical},
  53(38):385005, 2020.

\bibitem{weng2017}
Tongfeng Weng, Jie Zhang, Michael Small, and Pan Hui.
\newblock Multiple random walks on complex networks: A harmonic law predicts
  search time.
\newblock {\em Physical Review E}, 95(5):052103, 2017.

\bibitem{lawley2020networks}
Sean~D Lawley.
\newblock Extreme first-passage times for random walks on networks.
\newblock {\em Physical Review E}, 102(6):062118, 2020.

\bibitem{lawley2021pdmp}
Sean~D Lawley.
\newblock Extreme first passage times of piecewise deterministic markov
  processes.
\newblock {\em arXiv preprint arXiv:1912.03438}, 2019.

\bibitem{durrett2019}
R~Durrett.
\newblock {\em Probability: theory and examples}.
\newblock Cambridge university press, 2019.

\end{thebibliography}
\bibliographystyle{unsrt}

\end{document}